\documentclass[letterpaper]{amsart}       
\usepackage[foot]{amsaddr}
\usepackage{mathtools,tabularx,array}
\usepackage{color}
\usepackage[colorinlistoftodos]{todonotes}
\usepackage{doi}
\usepackage{graphicx}
\usepackage{subfigure}
\usepackage{float}
\usepackage{pgfplots}
\def\dsp{\displaystyle}
\def\uo{u^{\rm{OL}}}
\def\vo{v^{\rm{OL}}}
\def\uol{u_k^{\rm{OL}}}
\def\vol{v_k^{\rm{OL}}}
\def\uolj{u_j^{\rm{OL}}}

\def\DS{\displaystyle}
\newtheorem{theorem}{Theorem}[section]
\newtheorem{lemma}[theorem]{Lemma}
\newtheorem{proposition}[theorem]{Proposition}

\theoremstyle{definition}

\theoremstyle{remark}
\newtheorem{remark}[theorem]{Remark}

\numberwithin{equation}{section}

\begin{document}

\title{Dynamic Modeling of Nontargeted and Targeted Advertising Strategies in an Oligopoly 
}


\author{Chloe A. Fletcher}
\address[CAF]{Department of Computer Science, College of Charleston, Charleston, SC 29424}
\email[CAF]{fletcherca@g.cofc.edu}
\author{Jason S. Howell}
\address[JSH]{Department of Mathematics, College of Charleston, Charleston, SC 29424}
\email[JSH]{howelljs@cofc.edu}



\maketitle

\begin{abstract}
With the growing collection of sales and marketing data and depth of detailed knowledge of consumer habits and trends, firms are gaining the capability to discern customers of other firms from the potential market of uncommitted consumers.  Firms with this capability will be able to implement a strategy where the advertising effort towards customers of competing firms may differ from that towards uncommitted consumers.  In this work, dynamic models for advertising in an oligopoly setting with fixed total market size and sales decay are presented.  Two models are described in detail: a nontargeted model in which the advertising effort is the same for both categories of prospective customers, and a targeted model that gives firms the capability to allocate effort across the two categories differently.  In the differential game setting, open-loop and closed-loop Nash equilibrium strategies are derived for both models.  Several strategic questions that a firm may face when practicing targeted advertising on a fixed budget are discussed and addressed.

\end{abstract}

\section{Introduction}
\label{intro}

Recent advances in technology are enabling the collection of data that constructs a detailed profile of consumer behavior.  Firms who gain sales from advertising in a competitive setting will be faced with the challenge of selecting the ideal level of tailoring advertising campaigns based on consumer preferences. While analyzing  internet browsing history data to target advertisements is now common practice, the spread of customer data collection and analysis to other social and technology sectors is emerging as a powerful tool. 
 This will enable firms to segregate their potential customers into compartments: individuals in the market for a product or service offered by the firm, and current consumers of comparable products or services offered by competing firms.

With this capability in place, firms may wish to tailor advertising campaigns so that the effort (allocation of advertising expenditures) towards customers of other firms may differ from the effort towards uncommitted customers.  This may provide a more efficient distribution of advertising expenditures, and may be advantageous in a number of situations, including a highly competitive niche product market in which sales are gained primarily from customers of other firms, as well as the case where sales are gained primarily from the market potential.  This model may be useful in markets such as cable and satellite television service, wireless communication providers, security services, or even political campaigns.   This model can also represent a single firm that offers multiple mutually-exclusive tiers or levels within a single product line, such as lawn care service or wireless communication providers. In this case, the firm may decide (for example, through targeted mailings) to advertise more aggressively toward customers of the lower-tier services (with advertisements attempting to convince existing customers to upgrade) as opposed to potential customers, or vice-versa.





Dynamic models of advertising competition in the context of differential equations and dynamic games have been a rich focus of study over the last sixty years.  Kimball \cite{Kimball1957} employed the Lanchester \cite{Lanchester1916} combat model to represent two firms competing for market share via advertising efforts.  The model of Vidale and Wolfe \cite{Vidale1957} viewed the time rate of change of sales rate as a function of advertising expenditure and sales decay.  Several related models and contributions are found in \cite{Leitmann1978}, \cite{Feichtinger1983}, \cite{Chin1992}, \cite{Mesak1993}, \cite{jarrar2004}, \cite{bass2005}, \cite{Bass2005b}, \cite{Nguyen2006}, \cite{Jie07}, \cite{Jorgensen2010}, \cite{Dragone2010}, \cite{Krishna2010}, \cite{Liu2012},  \cite{Prasad2012}, and \cite{jorg15}.  
The reader is referred to \cite{Sethi1977}, \cite{Erickson1995}, and \cite{Huang2012} for comprehensive surveys of work in dynamic advertising models, as well as \cite{Jorgensen1982}, \cite{Moorthy1993}, \cite{Dockner2000}, \cite{Jorgensen2004} for advertising models in a differential game setting.  Of note, \cite{Hartl2005} extended the advertising model of \cite{Nerlove1962}, based on a stock of goodwill, to the situation where differing advertising policies are prescribed towards potential customers and existing customers.

Often the study of these dynamic models is undertaken within the setting of a non-cooperative differential game in which competing firms wish to employ strategies that maximize a profit objective function. In this scenario, firms set their advertising strategy by allocating a particular amount of effort (usually in the form of advertising expenditures), and this effort is considered to be the control variable.   A primary objective of the analysis of these games is to determine {\em Nash equilibrium}, which defines the strategies for all firms that are optimal in the sense that no single firm stands to gain from unilaterally modifying their own strategy (i.e., changing their own control).  This analysis can lead to an {\em open-loop} Nash equilibrium, in which the control variable depends only on time, or a {\em closed-loop} equilibrium, in which the control is also dependent on the current state of the system (the sales rates/market shares of all firms in the competitive market). This makes a closed-loop solution more attractive, as its dependence on the current state of the system allows for continual adjustment of strategy, as opposed to the open-loop strategy that is determined at the outset of the time horizon.  However, computation of a closed-loop strategy is generally more difficult than an open-loop strategy and often requires the solution of boundary-value problems in ordinary or partial differential equations.

As many of these models comprise systems of nonlinear differential equations, it is also useful to determine if steady states of the system exist and analyze their stability properties.  Further insight into strategic decision-making can also be found by making certain assumptions about model parameters and optimization with respect to certain variables.

The motivation for the work presented here arises mainly from \cite{Fruchter1999b}, \cite{Wang2001} and \cite{Fruchter2004}.  Fruchter \cite{Fruchter1999b} extends the Lanchester-based model in \cite{Fruchter1997} and \cite{Fruchter1999a} to an expanding market size, allowing for a decreasing market potential available to all firms.  This was extended to the oligopoly setting where each competing firm offered a line of products in \cite{Fruchter2001}.   Wang and Wu \cite{Wang2001} present an extension of the duopoly Vidale-Wolfe model to an expanding market with sales decay.  Closed-loop Nash equilibria were derived in both \cite{Fruchter1999b} and \cite{Wang2001}.  Fruchter and Zhang \cite{Fruchter2004} consider a model that divides customers into two categories: repeat customers and customers of other firms.  This model is analyzed in a duopoly setting in which firms allocate advertising effort differently between these customer categories.  

However, none of the aforementioned Lanchester-based models existing in the literature allow for differing advertising policies towards uncommitted consumers and competitors' customers.  
In this work two dynamic advertising models in an oligopoly setting are presented.   In a {\em nontargeted} scheme, a firm does not discern between competitors' customers and the untapped market potential.  This results in a model where the only decision, or control, is the advertising expenditure.  However, in
the new {\em targeted} scheme, a differentiation is made between competitors' customers and the market potential, and the firm must determine the best course of action, i.e., the best allocation of expenditures/effort toward these two groups.  Both models allow for sales/market share decay through cancellation.  J{\o}rgensen and Sigu\'{e} \cite{jorg15} recently presented a model which allows for a different advertising policy towards competitors' customers ({\em offensive} advertising) and the market potential ({\em generic} advertising), however is it assumed that generic advertising may benefit all firms participating in the market.

The nontargeted model we describe is an extension of the Lanchester-based oligopoly model of an expanding market in \cite{Fruchter1999b} to  include sales/market share decay.  In addition for allowing a variable (as opposed to nonincreasing) market potential, the  cancellation rate also allows for an interpretation of customer retention through a reciprocal relationship.  These models can easily be extended to include effort and effectiveness parameters of customer retention activities. 
For this model, closed-loop strategies, based on the solution of a two-point boundary value problem, are derived and shown to form Nash equilibria.  By extending \cite{Fruchter1999b} to allow for cancellation, the nontargeted model can also be viewed as an extension of the duopoly model in \cite{Wang2001} to the Lanchester sales-rate oligopoly setting, providing somewhat of a convergence of the Vidale-Wolfe and Lanchester models.

The nontargeted model is then modified by allowing for a different allocation of effort across the market potential and competitors' customers.  This targeted model  also allows for varying effectiveness to effort ratios for the two submarkets.  Indeed, the targeted model leads to a more complex mathematical problem and  is somewhat unwieldy for an increasing number of firms when describing closed-loop Nash equilibria, as well as the steady-state sales rates for given sets of parameters.  Thus the derivation of closed-loop strategies is limited to the duopoly case, and the steady-states are derived when there are two or three competing firms.  Additionally, analysis of the behavior of the model in specific settings provides insight when addressing questions that a firm may be faced with when determining an allocation of advertising effort across these two submarkets.  A particular application of interest is the situation when a candidate in a political campaign is able to discern the voters committed to his/her opponent and must determine the best allocation of advertising expenditures in a fixed budget.  Indeed, a recent announcement detailed a partnership between satellite television providers for developing a common database of subscriber data with the intent of implementation of addressable advertising for political campaigns.\footnote{See http://adage.com/article/media/dish-directv-team-addressable-ad-efforts/291303/}  In this situation a political party or candidate wishing
to purchase advertising may be faced with the question: it is better to advertise more aggressively to perceived supporters of an opponent or to undecided voters? 

To summarize the contributions presented here relative to related works, the Table \ref{tab:model} details this contribution of this manuscript and  related works.  In the table, ``Targeting'' represents whether or not the model allows for a differing allocation of effort across customer categories. A recent summary of all related dynamic models of advertising competition is presented in the Online Appendix of \cite{Huang2012}.
\renewcommand{\arraystretch}{1.1}
\begin{table}
\caption{Summary of Related Works \label{tab:model}}
{\begin{tabular}{>{\raggedright}m{0.85in}ccccccc}\hline\noalign{\smallskip}
	&	Fruchter 	&		&		&		&	Fruchter 	&J{\o}rgensen&		\\
	&	and	&		&		&	Wang and	&	 and 	&and	&This	\\
	&	Kalisch \cite{Fruchter1997}	&	Fruchter \cite{Fruchter1999a}	&	Fruchter \cite{Fruchter1999b}	&	Wu \cite{Wang2001}	&	Zhang \cite{Fruchter2004}	& Sigu\'{e}\cite{jorg15}&	work	\\\hline\noalign{\smallskip}
Model Type	&	Duopoly	&	Oligopoly	&	Oligopoly	&	Duopoly	&	Duopoly&	Duopoly	&	Oligopoly	\\[1ex]
Sales Decay	&	No	&	No	&	No	&	Yes	&	No&	No	&	Yes	\\[1ex]
Effort To 
Market Potential	&	Yes	&	Yes	&	Yes	&	Yes	&	No	&	Yes&	Yes		\\[1ex]
Effort To Competitors' Customers	&	Yes	&	Yes	&	Yes	&	Yes	&	Yes&	Yes	&	Yes	\\[1ex]
Targeting	&	No	&	No	&	No	&	No	&	Yes	&	Yes&	Yes	\\[1ex]
Time {Horizon}	&	Infinite	&	Infinite	&	Infinite	&	Finite	&	Infinite	&	Finite&	Finite	\\[1ex]
Open-loop NE	&	Yes	&	Yes	&	Yes	&	Yes	&	Yes	&	No&	Yes	\\[1ex]
Closed-loop NE	&	Yes	&	Yes	&	Yes	&	Yes	&	Yes&	Yes	&	Yes	\\\noalign{\smallskip}\hline
\end{tabular}}
\end{table}

\renewcommand{\arraystretch}{1.0}


\section{The Nontargeted Advertising Model}\label{sec:ns}
Consider an industry with $N\ge 2$ competing firms.  We assume that each firm uses advertising as their major marketing instrument to increase sales, which is done by  convincing other customers to switch firms or by gaining prospects from the market potential.

For 
$n=1,2,\ldots,N$, let $x_n(t)$ represent firm $n$'s market share at time $t$ with the assumption that $0\le x_n(t)\le 1$ for all time $t\in [0,\infty)$.  By convention we assume that the advertising expenditures result in diminishing returns in the way of advertising effort, thus we let $u^2_n(t)$ represent the advertising expenditure of firm $n$ so that $u_n(t)$ is the advertising effort.  The parameter $\rho_k$ the advertising effectiveness to effort ratio (thus $\rho_ku_k(t)$ represents the advertising effectiveness of firm $k$'s campaign).  
Let $m$ represents the total possible sales of the market, so it satisfies
\begin{equation}
m=\sum_{k=1}^n s_k(t)+\varepsilon(t),
\label{eq:mp}
\end{equation}
where  $\varepsilon(t)$ represents the market potential at time $t$.  When $m=1$, as will often be assumed in the applications discussed later, $s_k$ represents a market share.  Note that the market potential satisfies
\begin{equation}
\dot{\varepsilon}(t)=-\sum_{k=1}^n \dot{s}_k(t).
\label{eq:mpdot}
\end{equation}
The oligopoly advertising model with market expansion in \cite{Fruchter1999b} is given by
\begin{equation}
\dot{s}_k(t)=\underbrace{\rho_ku_k(t)\left(m-s_k(t)\right)}_{\substack{\text{Gain from advertising}\\\text{towards non-customers}}} -\underbrace{s_k(t)\sum_{j=1, j\ne k}^n\rho_ju_j(t)}_{\substack{\text{Loss from competitors'}\\\text{advertising efforts}}},
\label{eq:fru}
\end{equation}
for $k=1,\ldots,n$, as the first term on the right hand side represents the gain in sales from the market potential and customers of other firms, and the second term represents the sales lost to competitors' advertising efforts. It is assumed that $s_k(t), \rho_k, u_k(t)$ are all positive for all $t$.  To incorporate sales decay, or cancellation, into the model let $c_k(t)\ge0$ represent the rate at which firm $k$ loses customers to the market potential. Then the nontargeted advertising model is given by
\begin{equation}
\dot{s}_k(t)=\underbrace{\rho_ku_k(t)\left(m-s_k(t)\right)}_{\substack{\text{Gain from advertising}\\\text{towards non-customers}}} -\underbrace{s_k(t)\sum_{j=1, j\ne k}^n\rho_ju_j(t)}_{\substack{\text{Loss from competitors'}\\\text{advertising efforts}}}-\underbrace{c_k(t)s_k(t)}_{\substack{\text{Loss from}\\\text{cancellation}}}.
\label{eq:non1}
\end{equation} 
In light of \eqref{eq:mpdot} we have that the market potential satisfies the differential equation
\[\dot{\varepsilon}(t)=\sum_{k=1}^nc_k(t)s_k(t)-\varepsilon(t)\sum_{k=1}^n \rho_ku_k(t),\]
and, as opposed to the model of \cite{Fruchter1999b}, does not indicate that the market potential is monotonic decreasing.
The model \eqref{eq:non1} can also be written as
\begin{equation}
\dot{s}_k(t)=\rho_ku_k(t)m-s_k(t)\left(c_k(t)+\sum_{j=1}^n\rho_ju_j(t)\right).
\label{eq:non2}
\end{equation}
Given initial values $s_k(0)=s_k^0$, the system of first-order equations \eqref{eq:non2} for $k=1,\ldots,n$ has the solution
\begin{equation}
s_k(t) = \left[ s_k^0+  \int_0^tm\rho_ku_k(\tau)e^{\psi_k(\tau)}\,d\tau\right]e^{-\psi_k(t)},
\label{eq:non3}
\end{equation}
where
\[\psi_k(t) = \int_0^t\left(c_k(\tau)+\sum_{j=1}^n\rho_ju_j(\tau)\right)\,d\tau.\]
\begin{remark}\label{rem:nonzero}
While there is no guarantee that $\lim_{t\to\infty}s_k(t) >0$, we have that $s_k(T)>0$ for any finite time horizon $[0,T]$.  Thus a finite time horizon is considered in the discussion of the differential game and Nash equilibrium in Section \ref{sec:nsnash}.
\end{remark}

\begin{remark}
When \eqref{eq:non1} is written in terms of firm $k$'s market share $x_k=s_k/m$, the duopoly case reduces to the model of \cite{Wang2001}.
\end{remark}

\subsection{Nash Equilibrium}\label{sec:nsnash}
In this section we discuss the derivation of Nash equilibrium strategies for the nontargeted advertising model over a finite time horizon.  
Assume the discount rate $r$ is uniform for all firms and let $x_k=s_k/m$. Then the profit objective function is given by
\begin{equation}
\Pi_k = \int_0^T \left(Q_kx_k(t)-u_k^2(t)\right)e^{-rt}\,dt,
\label{eq:nsprof}
\end{equation}
where $Q_k=q_km$ and $q_k$ is the gross profit rate.
The differential game is characterized as follows: the problem of oligopolist $k$, $k=1,\ldots,n$ is to find the control $u_k$ such that
\begin{equation}
\max_{u_k} \Pi_k(u_1,\ldots,u_k,\ldots,u_n)=\max_{u_k}
\int_0^T \left(Q_kx_k(t)-u_k^2(t)\right)e^{-rt}\,dt,
\label{eq:nsgame1}
\end{equation}
subject to
\begin{equation}
\dot{x}_k=\rho_ku_k-x_k\left(c_k+\sum_{j=1}^n\rho_ju_j\right),\qquad x_k(0)=x_k^0.
\label{eq:nsgame2}
\end{equation}
The control $u_k$ is admissible provided $0\le u_k<\infty$.  For a Nash equilibrium closed-loop strategy, we must find $u_k^*(t,x_1,\ldots,x_n,x_1^0,\ldots,x_n^0)$ such that
\begin{equation}
\Pi_k(u_1^*,\ldots,u_k^*,\ldots,u_n^*) \ge \Pi_k(u_1^*,\ldots,u_{k-1}^*,u_k,u_{k+1}^*\ldots,u_n^*)\quad \forall u_k,\quad k=1,\ldots,n.
\label{eq:nsgame3}
\end{equation}

\def\uol{u_k^{\text{OL}}}
\begin{theorem}\label{thm:nsnash}
Assume $x_k^*$ and $\varphi_k$, $k=1,\ldots,n$, solve the following two-point boundary value problem of $2n$ equations: 
\begin{alignat}{1}
\dot{x}_k&= \frac{1}{2}\left[\rho^2_kQ_k\varphi_k(t)(1-x_k)e^{rt}-x_k\left(2c_k+\sum_{j=1}^n\rho_j^2Q_j\varphi_j(t)e^{rt}(1-x_j)\right)\right],\, x_k(0)=x_k^0,
\label{eq:thm11}\\
\dot{\varphi_k}&=c_k\varphi_k(t)+\frac{1}{2}\varphi_k(t)\sum_{j=1}^n\rho_j^2Q_j\varphi_j(t)(1-x_j)e^{rt} - e^{-rt}, \quad\varphi_k(T)=0.\label{eq:thm12}
\end{alignat}
Then the functions
\begin{equation}
u_k^{*} = \frac{1}{2}\rho_kQ_k\varphi_k(t)e^{rt}(1-x_k),\qquad k=1\ldots,n,
\label{eq:nsclosed}
\end{equation}
where the $x_k$ satisfy \eqref{eq:nsgame2},
form a global Nash equilibrium closed-loop strategy for the differential game \eqref{eq:nsgame1}--\eqref{eq:nsgame2}, and the functions
\begin{equation}
u_k^{\text{OL}} = \frac{1}{2}\rho_kQ_k\varphi_k(t)e^{rt}(1-x^*_k),\qquad k=1\ldots,n,
\label{eq:nso}
\end{equation}
form a open-loop strategy for the differential game \eqref{eq:nsgame1}--\eqref{eq:nsgame2}.
\end{theorem}

\begin{proof}
The argument is an extension of the proof of Theorem 1 in \cite{Fruchter1999b} and utilizes the technique introduced in \cite{Fruchter1997} to show the closed-loop strategies are optimal.
The current value Hamiltonian of firm $k$ is given by
\begin{equation}
H_k(x_k,u_1,\ldots,u_n,\lambda_k,t) = Q_kx_k-u_k^2 +\lambda_{k,k}\dot{x}_k+\sum_{\substack{j=1\\j\ne k}}^n\lambda_{k,j}\dot{x}_j
\end{equation}
where $\lambda_{k,j}$ are the costate variables.  We have that $\partial H_k/\partial u_k=0$ precisely when
\begin{equation}
-2u_k+\rho_k\left(\lambda_{k,k}(1-x_k)  -\sum_{\substack{j=1\\j\ne k}}^n\lambda_{k,j}x_j \right)=0. 
\label{eq:nsgame4}
\end{equation}
for $k=1\ldots,n$.  The optimality conditions are then given by \eqref{eq:nsgame4} and
\begin{eqnarray}
\dot{\lambda}_{k,k} &=&r\lambda_{k,k}-\frac{\partial H_k}{\partial x_k} =r\lambda_{k,k}-Q_k+\lambda_{k,k}\left(c_k+\sum_{j=1}^n\rho_ju_j\right),\qquad \forall k=1,\ldots,n,
\label{eq:nsgame5}\\
\dot{\lambda}_{k,j} &=&r\lambda_{k,j}-\frac{\partial H_k}{\partial x_j} =\lambda_{k,j}\left(r+c_j+\sum_{i=1}^n\rho_iu_i\right), \qquad\forall k,j=1,\ldots,n, j\ne k,
\label{eq:nsgame6}
\end{eqnarray}
with transversality conditions 
\begin{equation}
\lambda_{k,j}(T)x_k(T)=0,\qquad \forall k,j=1,\ldots,n,
\label{eq:nsgame7}
\end{equation}
that imply $\lambda_{k,j}(T)=0$ for all $k$ and $j$ (see Remark \ref{rem:nonzero}).
Now \eqref{eq:nsgame6} and \eqref{eq:nsgame7} imply that $\lambda_{k,j}=0$ for all $j\ne k$.  Then allowing $\lambda_k:=\lambda_{k,k}$, \eqref{eq:nsgame4} implies
\begin{equation}
u_k=\frac{1}{2}\rho_k \lambda_k(1-x_k).
\label{eq:nsgame8}
\end{equation}
This is the same control $u_k$ found in \cite{Fruchter1999b}.  Then, substituting \eqref{eq:nsgame8} into $\dot{x}_k=\partial H_k/\partial \lambda_k$ and \eqref{eq:nsgame5} we obtain the two-point boundary value problem with $2n$ equations given by: for $k=1,\ldots,n$, 
\begin{alignat}{3}
\dot{x}_k&= \frac{1}{2}\left[\rho^2_k\lambda_k(1-x_k)-x_k\left(2c_k+\sum_{j=1}^n\rho_j^2\lambda_j(1-x_j)\right)\right],\qquad &x_k(0)&=&x_k^0,
\label{eq:tpbvp1}\\
\dot{\lambda}_k&=\lambda_k \left(r+c_k+\frac{1}{2}\sum_{j=1}^n\rho_j^2\lambda_j(1-x_j)\right)-Q_k. &\lambda_k(T)&=&0.\label{eq:tpbvp2}
\end{alignat}
Make a change of variable via the definition
$\varphi_k(t)=e^{-rt}\lambda_k(t)/Q_k$,  $k=1,\ldots,n$.  As opposed to \cite{Fruchter1999b}, in which the system of $2n$ equations can be reduced to $n+1$ equations, \eqref{eq:tpbvp1}--\eqref{eq:tpbvp2} cannot.  However, the substitution for $\lambda_k$ is still useful as will be demonstrated later.
Then \eqref{eq:tpbvp1}--\eqref{eq:tpbvp2} can be written as 
\begin{alignat}{3}
\dot{x}_k&= \frac{1}{2}\left[\rho^2_kQ_k\varphi(t)(1-x_k)e^{rt}-x_k\left(2c_k+\sum_{j=1}^n\rho_j^2Q_j\varphi(t)e^{rt}(1-x_j)\right)\right],\quad &x_k(0)&=&x_k^0,
\label{eq:tpbvp3}\\
\dot{\varphi_k}&=c_k\varphi_k(t)+\frac{1}{2}\varphi_k(t)\sum_{j=1}^n\rho_j^2Q_j\varphi_j(t)(1-x_j)e^{rt} - e^{-rt}, &\varphi_k(T)&=&0.\label{eq:tpbvp4}
\end{alignat}
Let $x_k^*$, $\varphi_k$, $k=1,\ldots,n$ solve \eqref{eq:tpbvp3}--\eqref{eq:tpbvp4} and define $u_k^{*}$ and $u_k^{\text{OL}}$ as in \eqref{eq:nsclosed} and \eqref{eq:nso}, respectively.
The objective is to find an expression such that, when added to $\Pi_k$, demonstrates that \[\Pi_k(u_1^*,\ldots,u_{k-1}^*,u_k^*,u_{k+1}^*,\ldots,u_n^*) \ge \Pi_k(u_1^*,\ldots,u_{k-1}^*,u_k, u_{k+1}^*,\ldots,u_n^*).\]
We have
\[
\int_0^T\frac{d}{dt}\left(Q_k\varphi_k(t)x_k(t)\right)\,dt = Q_k\varphi_k(t)x_k(T)\bigg|_{0}^{T} = -Q_k\varphi_k(0)x_k^0,
\]
so that
\begin{equation}
Q_k\varphi_k(0)x_k^0+\int_0^T\frac{d}{dt}\left(Q_k\varphi_k(t)x_k(t)\right)\,dt = 0
\label{eq:zs1}
\end{equation}
Using \eqref{eq:nsgame2}, \eqref{eq:tpbvp4}, and \eqref{eq:nsclosed}, we have
\begin{align}
\frac{d}{dt}&\left(Q_k\varphi_k x_k\right) =
Q_k\dot{\varphi}_kx_k+Q_k\varphi_k\dot{x}_k\nonumber\\
&=Q_kx_k\left(c_k\varphi_k+\frac{1}{2}\varphi_k\sum_{j=1}^n\rho_j^2Q_j\varphi_j(1-x^*_j)e^{rt} - e^{-rt} \right)+Q_k\varphi_k\left(\rho_ku_k-x_k\left(c_k+\sum_{j=1}^n\rho_ju_j\right)\right)\nonumber\\
&=Q_kx_kc_k\varphi_k-Q_kx_ke^{-rt} +\frac{1}{2}Q_k\varphi_kx_k\sum_{j=1}^n\rho_j^2Q_j\varphi_j(1-x^*_j)e^{rt}
+Q_k\varphi_k\rho_ku_k\nonumber\\
&\hskip 0.4in-Q_k\varphi_k x_k c_k - Q_k\varphi_k x_k\sum_{j=1}^n\rho_ju_j\nonumber\\
&=\frac{1}{2}Q_k\varphi_kx_k\sum_{j=1}^n\rho_j^2Q_j\varphi_j(1-x^*_j)e^{rt}
+Q_k\varphi_k\rho_ku_k - Q_k\varphi_k x_k\sum_{j=1}^n\rho_ju_j-Q_kx_ke^{-rt}\nonumber\\
&=Q_k\varphi_k\rho_k(u_k-\uol+\uol)+Q_k\varphi_k x_k\sum_{j=1}^n\rho_j \left(\underbrace{\frac{1}{2}\rho_jQ_j\varphi_j e^{rt}(1-x^*_j)}_{=\uolj}
 - u_j\right)-Q_kx_ke^{-rt}\nonumber\\
&=Q_k\varphi_k\rho_k(u_k-\uol)+Q_k\varphi_k\rho_k\uol-Q_k\varphi_k x_k\sum_{j=1}^n\rho_j \left(u_j-\uolj
\right)-Q_kx_ke^{-rt}\nonumber\\ 
&=Q_k\varphi_k\rho_k(u_k-\uol)+Q_k\varphi_k\rho_k\uol-Q_kx_ke^{-rt}-Q_k\varphi_k x_k\left(\rho_k(u_k-\uol) + \sum_{\substack{j=1\\j\ne k}}^n\rho_j \left(u_j-\uolj
\right)   \right) \nonumber\\
&=(u_k-\uol)Q_k\varphi_k\rho_k(1-x_k)+Q_k\varphi_k\rho_k\uol-Q_kx_ke^{-rt}-Q_k\varphi_k x_k\sum_{\substack{j=1\\j\ne k}}^n\rho_j \left(u_j-\uolj
\right) \nonumber\\
&=(u_k-\uol)\underbrace{\frac{1}{2}\rho_kQ_k\varphi_k e^{rt}(1-x_k)}_{=u_k^*}(2e^{-rt})+Q_k\varphi_k\rho_k\uol-Q_kx_ke^{-rt}\nonumber\\
&\hskip 0.4in +2\left(\frac{1}{2}\right)\left(\frac{\rho_k}{\rho_k}\right)Q_k\varphi_k\left(e^{rt}\right)\left(e^{-rt}\right) (1-x_k-1)\sum_{\substack{j=1\\j\ne k}}^n\rho_j \left(u_j-\uolj
\right) \nonumber\\
&=\left[2(u_k-\uol)u_k^*+\rho_kQ_k\varphi_k e^{rt}\uol-Q_kx_k\right]e^{-rt}\nonumber\\
&\hskip 0.4in +\left[\frac{2}{\rho_k}\left(\underbrace{\frac{1}{2}\rho_kQ_k\varphi_k e^{rt}(1-x_k)}_{u_k^*}\right)\sum_{\substack{j=1\\j\ne k}}^n\rho_j \left(u_j-\uolj
\right)-Q_k\varphi_k e^{rt}\sum_{\substack{j=1\\j\ne k}}^n\rho_j \left(u_j-\uolj
\right)\right]e^{-rt} \nonumber\\
&=\biggl[2u_ku_k^*-2\uol u_k^*+\rho_kQ_k\varphi_k e^{rt}\uol-Q_kx_k 
\nonumber\\&\hskip 0.4in 
+\frac{2u_k^*}{\rho_k}\sum_{\substack{j=1\\j\ne k}}^n\rho_j \left(u_j-\uolj
\right)-Q_k\varphi_k e^{rt}\sum_{\substack{j=1\\j\ne k}}^n\rho_j \left(u_j-\uolj
\right)\biggr]e^{-rt}. \label{eq:zs2}
\end{align}
Then \eqref{eq:zs1} and \eqref{eq:zs2} together give
\begin{align}
\Pi_k&(u_1,\ldots,u_n)= \Pi_k(u_1,\ldots,u_n)+0\nonumber\\
&= \Pi_k(u_1,\ldots,u_n)+Q_k\varphi(0)x_k^0+\int_0^T\frac{d}{dt}\left(Q_k\varphi_k(t)x_k(t)\right)\,dt\nonumber\\
&=Q_k\varphi_k(0)x_k^0+ \int_0^T\left(Q_kx_k-u_k^2\right)e^{-rt}\,dt+\int_0^T\frac{d}{dt}\left(Q_k\varphi_k(t)x_k(t)\right)\,dt\nonumber\\
&=Q_k\varphi_k(0)x_k^0+ \int_0^T\left[ 2u_ku_k^*-2\uol u_k^*-u_k^2\right]e^{-rt}\,dt
\nonumber\\
&\, +\int_0^T\left[ \rho_kQ_k\varphi_k e^{rt}\uol+\frac{2u_k^*}{\rho_k}\sum_{\substack{j=1\\j\ne k}}^n\rho_j \left(u_j-\uolj
\right)-Q_k\varphi_k e^{rt}\sum_{\substack{j=1\\j\ne k}}^n\rho_j \left(u_j-\uolj
\right) \right]e^{-rt}\,dt.\label{eq:zs3}
\end{align}
Note that the second integral is independent of the $k$th argument of $\Pi$.  Then we have
\begin{multline}
\Pi_k(u_1^*,\ldots,u_{k-1}^*,u_k^*,u_{k+1}^*,\ldots,u_n^*) = Q_k\varphi_k(0)x_k^0+ \int_0^T\left[ {u_k^*}^2-2\uol u_k^*\right]e^{-rt}\,dt\\
+\int_0^T\left[ \rho_kQ_k\varphi_k e^{rt}\uol+\frac{2u_k^*}{\rho_k}\sum_{\substack{j=1\\j\ne k}}^n\rho_j \left(u_j^*-\uolj
\right)-Q_k\varphi_k e^{rt}\sum_{\substack{j=1\\j\ne k}}^n\rho_j \left(u_j^*-\uolj
\right) \right]e^{-rt}\,dt,
\label{eq:zs4}
\end{multline}
and
\begin{multline}
\Pi_k(u_1^*,\ldots,u_{k-1}^*,u_k,u_{k+1}^*,\ldots,u_n^*) = Q_k\varphi_k(0)x_k^0+ \int_0^T\left[ 2u_ku_k^*-2\uol u_k^*-u_k^2\right]e^{-rt}\,dt\\
+\int_0^T\left[ \rho_kQ_k\varphi_k e^{rt}\uol+\frac{2u_k^*}{\rho_k}\sum_{\substack{j=1\\j\ne k}}^n\rho_j \left(u_j^*-\uolj
\right)-Q_k\varphi_k e^{rt}\sum_{\substack{j=1\\j\ne k}}^n\rho_j \left(u_j^*-\uolj
\right) \right]e^{-rt}\,dt.
\label{eq:zs5}
\end{multline}
Subtracting \eqref{eq:zs5} from \eqref{eq:zs4} gives
\begin{multline}
\Pi_k(u_1^*,\ldots,u_{k-1}^*,u_k^*,u_{k+1}^*,\ldots,u_n^*) -
\Pi_k(u_1^*,\ldots,u_{k-1}^*,u_k,u_{k+1}^*,\ldots,u_n^*)\\
=
\int_0^T\left[{u_k^*}^2-2u_ku_k^*+u_k^2\right]e^{-rt}\,dt
=\int_0^T\left({u_k^*}-u_k\right)^2e^{-rt}\,dt\ge 0.
\label{eq:zs6}
\end{multline}
This shows that the closed-loop strategy \eqref{eq:nsclosed} is a Nash equilibrium.
\end{proof}

We remark that, if $c_k=0$ for all $k$, then $\varphi_k(t)=\varphi(t)$ is independent of $k$ and the two-point boundary value problem of $2n$ equations \eqref{eq:thm11}--\eqref{eq:thm12} reduces to the same $n+1$ equations of Theorem 1 of \cite{Fruchter1999b} and therefore yields the same closed-loop strategy.

While computation of analytic solutions to the two-point boundary value problem in \eqref{eq:thm11}--\eqref{eq:thm12}  is generally infeasible, numerical solutions can easily be obtained.  In Figure \ref{fig:nstp}, two different solutions are computed for the duopoly case.  In both computations, the discount rate is taken to be $r=0.1$ and the profit rates $Q_1=Q_2=1$.  We assume that firm 2's effectiveness to effort ratio is 20\% better than firm 1 ($\rho_1=1.0, \rho_2=1.2$) and that both firms have an identical initial market share of 40\%.  The figure on the left is the solution for $c_1=c_2=0$, and the figure on the right is with $c_1=0.1, c_2=0.2$.  It is easy to see that in the absence of sales decay, the market potential tends to zero, while there is a substantial market potential in the case with nonzero sales decay.   In both cases, the superior effectiveness of firm 2's campaign leads to a leading market share, even when the sales decay rate is higher.
\begin{figure}[ht]
\caption{Solutions to \eqref{eq:thm11}--\eqref{eq:thm12} and closed-loop strategies  \eqref{eq:nsclosed}\label{fig:nstp}}
{\subfigure[$c_1=c_2=0$]{%
\includegraphics[width=0.65\textwidth]{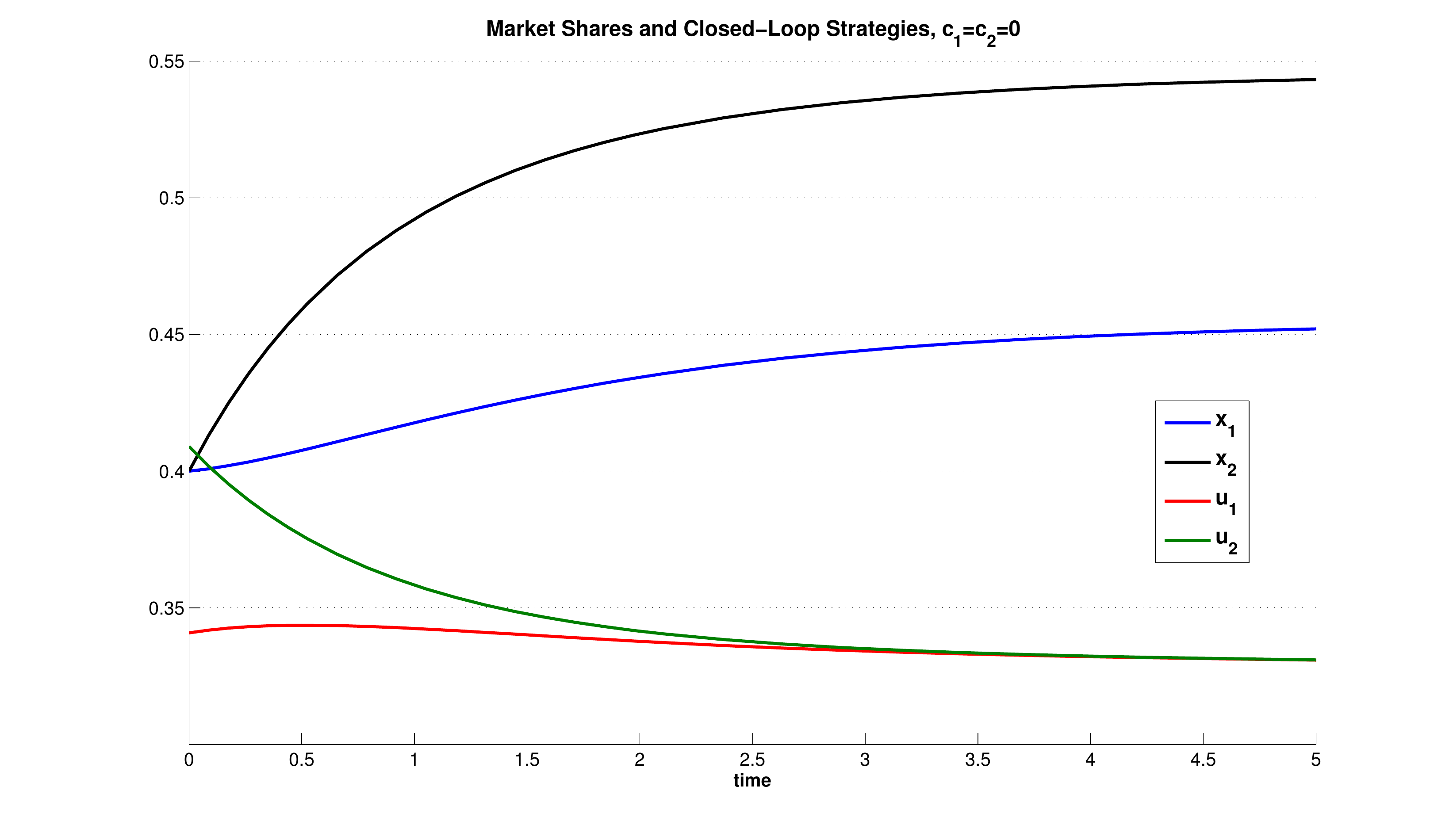}
\label{fig:nstp1}}
\quad
\subfigure[$c_1=0.1$, $c_2=0.2$]{%
\includegraphics[width=0.65\textwidth]{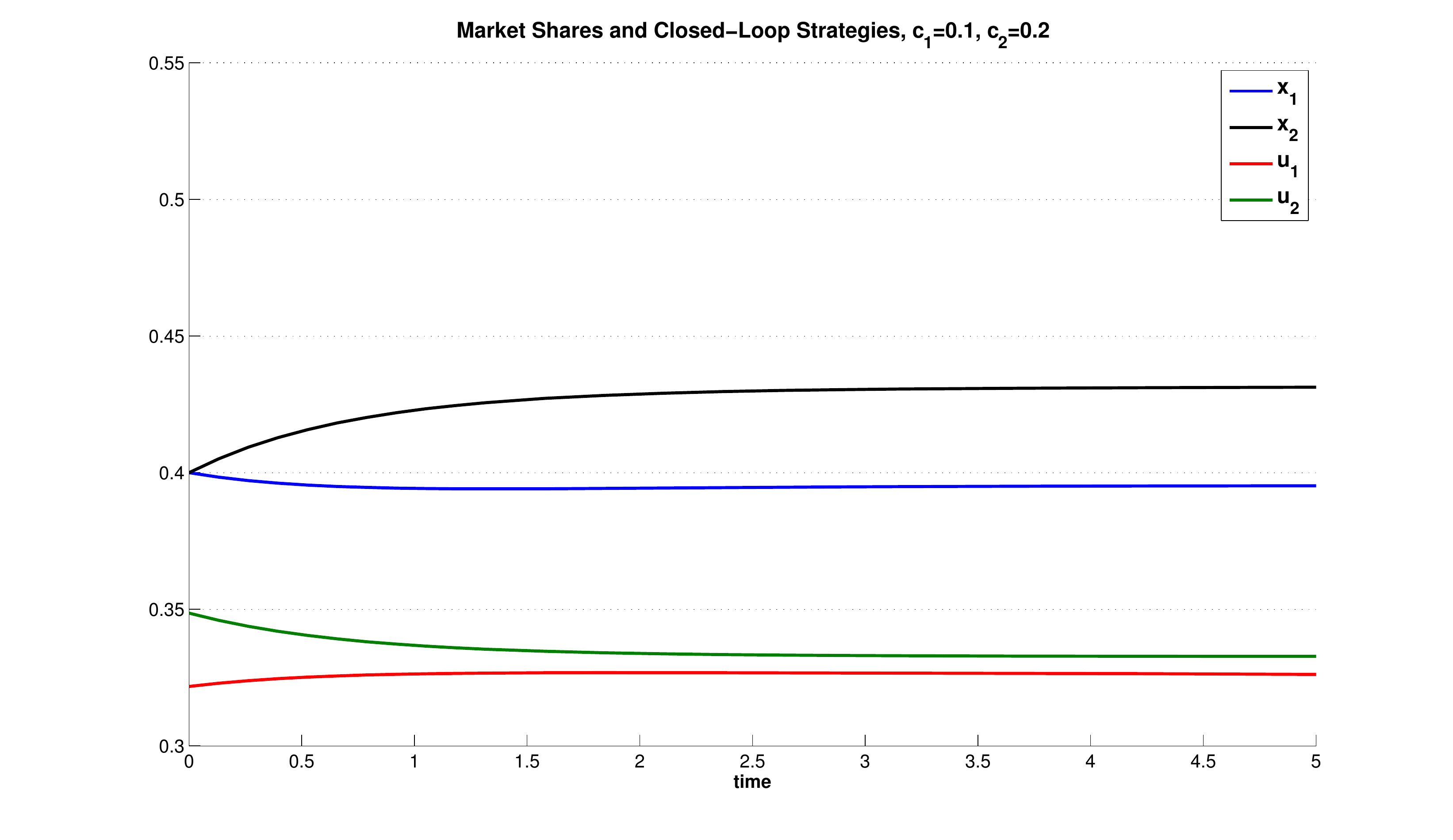}
\label{fig:nstp2}}}
\end{figure}

\subsection{Analysis of Steady States of the Nontargeted Model}\label{sec:nssteady}

For analysis of steady states of the nontargeted advertising model, the parameters $u_i$, $\rho_i$, and $c_i$ are all assumed to be positive constants. For notational simplicity the dependence on $t$ of $s_i$ will be suppressed.  The steady state solutions are found by setting the right-hand sides of \eqref{eq:non2} to zero for $k=1,\ldots,n$, i.e., 
\begin{equation}
\rho_ku_k(t)m-s_k(t)\left(c_k(t)+\sum_{j=1}^n\rho_ju_j(t)\right)=0,\qquad k=1,\ldots,n.
\label{eq:nseq}
\end{equation}

\begin{lemma}\label{lem:nsequil}
Let 
\begin{equation}
U_n = \sum_{j=1}^n\rho_ju_j.
\label{eq:Udef}
\end{equation}
Then the system \eqref{eq:nseq} has a unique solution given by
\begin{equation}
s_k=\frac{m\rho_ku_k}{c_k+U_n}.
\label{eq:nseq1}
\end{equation}
\end{lemma}

Existence and uniqueness of the solution are guaranteed by the assumptions on the parameters.  In this case, it is clear that the long-term sales of firm $k$ are dependent on its cancellation rate and the advertising effectiveness of all firms - an increase in advertising effectiveness by any other firm or an increase in cancellation rate result in a decrease in sales.  The equilibrium market potential is given by
\[\varepsilon = m-\sum_{k=1}^n s_k = m\left(1-\sum_{k=1}^n\frac{\rho_ku_k}{c_k+U_n}  \right).\]
As $\dot{s}_k$ depends only on the parameters and $s_k$ itself, the Jacobian $J$ of the system \eqref{eq:non2} for $k=1,\ldots,n$ is diagonal with each diagonal entry
\begin{equation}
J_{ii} = -c_i-\dsp\sum_{j=1}^n\rho_ju_j,
\label{eq:nsjac}
\end{equation}
which implies that the eigenvalues of the system are always negative.  Thus the equilibrium point \eqref{eq:nseq1} is asymptotically stable.


\section{The Targeted Advertising Model}\label{sec:s}
The targeted advertising model is constructed by allowing firm $k$ to employ differing advertising campaigns towards the market potential and customers of other firms.  
Let $v_k(t)>0$ be the advertising effort of firm $k$ towards the market potential and let $\sigma_k>0$ be the corresponding effectiveness to effort ratio.  Then the targeted advertising model of firm $k$'s sales at time $t$ is
\begin{equation}
\dot{s}_k(t)=\underbrace{\sigma_kv_k(t)\varepsilon(t)}_{\substack{\text{Gain from}\\\text{market potential}}} +\underbrace{\rho_ku_k(t)\sum_{j=1, j\ne k}^ns_j(t)}_{\substack{\text{Gain from}\\\text{competitors' customers}}}  -\underbrace{s_k(t)\left(c_k(t)+\sum_{j=1, j\ne k}^n\rho_ju_j(t)\right)}_{\substack{\text{Loss from competitors' advertising }\\\text{efforts and cancellation}}},
\label{eq:sel}
\end{equation}
where the first term on the right hand side represents the gain in sales from the market potential, the second term represents the gain in sales from customers of other firms, and the third term represents the decrease in sales from cancellations and the cumulative efforts of other firms.
Using the definition of the market potential \eqref{eq:mp} we can write \eqref{eq:sel} as
\begin{equation}
\dot{s}_k(t)=\sigma_kv_k(t)\left(m-\sum_{j=1}^ns_j(t)\right) +\rho_ku_k(t)\sum_{j=1, j\ne k}^ns_j(t)  -s_k(t)\left(c_k(t)+\sum_{j=1, j\ne k}^n\rho_ju_j(t)\right).
\label{eq:sel2}
\end{equation}
It is also useful to view \eqref{eq:sel2} as
\begin{equation}
\dot{s}_k(t)=\sigma_kv_k(t)m+\left(\rho_ku_k(t)-\sigma_kv_k(t)\right)\sum_{j=1, j\ne k}^ns_j(t)  -s_k(t)\left(\sigma_kv_k(t)+c_k(t)+\sum_{j=1, j\ne k}^n\rho_ju_j(t)\right).
\label{eq:sel3}
\end{equation}

\begin{remark}
If $u_k=v_k$ and $\rho_k=\sigma_k$ for all $k$ then \eqref{eq:sel3} reduces to the nontargeted advertising model \eqref{eq:non1}.
\end{remark}

\subsection{Nash Equilibrium}\label{sec:snash}
As in Section \ref{sec:nsnash}, we let $x_k=s_k/m$ and consider the profit objective function given by
\begin{equation}
\Pi_k = \int_0^T \left(Q_kx_k(t)-(u_k^2(t)+v_k^2(t))\right)e^{-rt}\,dt,
\label{eq:sprof}
\end{equation}
where $Q_k=q_km$ and $q_k$ is the gross profit rate.
The differential game is characterized as follows: the problem of oligopolist $k$, $k=1,\ldots,n$ is to find the control $u_k$ such that
\begin{equation}
\max_{u_k, v_k} \Pi_k((u_1,v_1),\ldots,(u_n,v_n))=\max_{u_k, v_k}
\int_0^T \left(Q_kx_k(t)-(u_k^2(t)+v_k^2(t))\right)e^{-rt}\,dt,
\label{eq:sgame1}
\end{equation}
subject to
\begin{equation}
\dot{x}_k=\sigma_kv_k\left(1-\sum_{j=1}^nx_j\right) +\rho_ku_k\sum_{j=1, j\ne k}^nx_j  -x_k\left(c_k+\sum_{j=1, j\ne k}^n\rho_ju_j\right).
\label{eq:sgame2}
\end{equation}
The controls $u_k$ and $v_k$ are admissible provided $0\le u_k, v_k <\infty$.

For ease of presentation, we give the closed-loop Nash equilibrium strategies for the targeted duopoly modeled by
\begin{eqnarray}
\dot{x}_1&=& \sigma_1v_1(1-x_1-x_2)+\rho_1u_1x_2-x_1(c_1+\rho_2u_2),\label{eq:snash1}\\
\dot{x}_2&=& \sigma_2v_2(1-x_1-x_2)+\rho_2u_2x_1-x_2(c_2+\rho_1u_1).\label{eq:snash2}
\end{eqnarray}
In this case, a global Nash equilibrium strategy is a pair $(u_1^*, v_1^*), (u_2^*, v_2^*)$ such that
\begin{eqnarray}
\Pi_1((u_1^*, v_1^*),((u_2^*, v_2^*)) &\ge& \Pi_1((u_1, v_1^*),((u_2^*, v_2^*)),\label{eq:gne1}\\
\Pi_1((u_1^*, v_1^*),((u_2^*, v_2^*)) &\ge& \Pi_1((u_1^*, v_1),((u_2^*, v_2^*)),\label{eq:gne2}\\
\Pi_2((u_1^*, v_1^*),((u_2^*, v_2^*)) &\ge& \Pi_2((u_1^*, v_1^*),((u_2, v_2^*)),\label{eq:gne3}\\
\Pi_2((u_1^*, v_1^*),((u_2^*, v_2^*)) &\ge& \Pi_2((u_1^*, v_1^*),((u_2^*, v_2)),\label{eq:gne4}
\end{eqnarray}
for all admissible $u_1, u_2, v_1, v_2$.  The closed-loop Nash equilibrium
strategies are found given the solvability of a two-point boundary value problem with six equations and six unknowns.

\begin{theorem}\label{thm:snash}
Let $x_k^*$, $\varphi_{k,j}$, $k, j=1,2$, satisfy the following two-point boundary value problem:
\begin{eqnarray}
\dot{x}_k^*&=& \frac{1}{2}\sigma_k^2Q_k\varphi_{k,k}(1-x_1^*-x_2^*)^2e^{rt} +\frac{1}{2}\rho_k^2Q_k(\varphi_{k,k}-\varphi_{k,j}){x_j^*}^2e^{rt} \nonumber\\&&\hskip 1.0in-\,x_k^*\left(c_k+\frac{1}{2}\rho_j^2Q_j(\varphi_{j,j}-\varphi_{j,k})x_k^*e^{rt}\right),\label{eq:sbvp1}\\
\dot{\varphi}_{k,k}&=&
\varphi_{k,k}\left(c_k+\frac{1}{2}\sigma_k^2Q_k\varphi_{k,k}(1-x_1^*-x_2^*)^2e^{rt} \right)+
\frac{1}{2}\sigma_j^2Q_j\varphi_{k,j}\varphi_{j,j}(1-x_1^*-x_2^*)e^{rt}
 \nonumber \\
&&\hskip 1.0in+\,\frac{1}{2}(\varphi_{k,k}-\varphi_{j,j})\rho_j^2Q_j(\varphi_{j,j}-\varphi_{j,k})x_j^*e^{rt}-e^{-rt},\label{eq:sbvp2}\\
\dot{\varphi}_{k,j}&=&
\frac{1}{2}\sigma_k^2Q_k\varphi_{k,k}^2(1-x_1^*-x_2^*)e^{rt}
-\frac{1}{2}\rho_k^2Q_k(\varphi_{k,k}-\varphi_{k,j})^2x_j^*e^{rt} \nonumber\\
&&\hskip 1.0in +\,
\varphi_{k,j}\left(c_j+\frac{1}{2}\sigma_j^2Q_j\varphi_{j,j}(1-x_1^*-x_2^*)e^{rt} \right),\label{eq:sbvp3}
\end{eqnarray}
where $x_k(0)=x_k^0$, $\varphi_{k,j}(T)=0$, $j,k=1,2$. 
Then 
\begin{equation}
u_k^*=\frac{1}{2}Q_k\rho_k(\varphi_{k,k}-\varphi_{k,j})x_je^{rt},\qquad
v_k^*=\frac{1}{2}Q_k\sigma_k\varphi_{k,k}(1-x_1-x_2)e^{rt},
\label{eq:scl}
\end{equation}
for $k, j=1, 2$, $j\ne k$, where the  $x_k$ satisfy \eqref{eq:snash1}--\eqref{eq:snash2},
form global closed-loop Nash equilibrium strategies for \eqref{eq:sgame1}--\eqref{eq:sgame2} (with $n=2$).
\end{theorem}

\begin{proof}
The current value Hamiltonians of firms $1$ and $2$ are given by
\begin{eqnarray}
H_1 &=& Q_1x_1-u_1^2-v_1^2 +\lambda_{1,1}\dot{x}_1+\lambda_{1,2}\dot{x}_2,\\
H_2 &=& Q_2x_2-u_2^2-v_2^2 +\lambda_{2,2}\dot{x}_2+\lambda_{2,1}\dot{x}_1.
\end{eqnarray}
where $\lambda_{k,j}$ are the costate variables.  Using \eqref{eq:snash1}--\eqref{eq:snash2} we have 
\begin{eqnarray}
-2u_1+\lambda_{1,1}\rho_1x_2-\lambda_{1,2}\rho_1x_2 =0\qquad &\Longrightarrow&\qquad u_1=\frac{1}{2}\rho_1(\lambda_{1,1}-\lambda_{1,2})x_2, \label{eq:u1def}\\
-2v_1+\lambda_{1,1}\sigma_1(1-x_1-x_2)=0\qquad &\Longrightarrow &\qquad v_1=\frac{1}{2}\sigma_1\lambda_{1,1}(1-x_1-x_2),\label{eq:v1def}\\
-2u_2+\lambda_{2,2}\rho_2x_1-\lambda_{2,1}\rho_2x_1 =0\qquad &\Longrightarrow&\qquad u_2=\frac{1}{2}\rho_2(\lambda_{2,2}-\lambda_{2,1})x_1, \label{eq:u2def}\\
-2v_2+\lambda_{2,2}\sigma_2(1-x_1-x_2)=0\qquad &\Longrightarrow &\qquad v_2=\frac{1}{2}\sigma_2\lambda_{2,2}(1-x_1-x_2).\label{eq:v2def}
\end{eqnarray}
The optimality conditions are then given by \eqref{eq:u1def}--\eqref{eq:v2def} and the four equations
\begin{eqnarray}
\dot{\lambda}_{1,1} &=&\lambda_{1,1}\left(r+c_1+\sigma_1 v_1+\rho_2u_2\right)+\lambda_{1,2}\left(\sigma_2v_2-\rho_2u_2\right)-Q_1,
\label{eq:l11dot}\\
\dot{\lambda}_{1,2} &=&\lambda_{1,1}\left(\sigma_1v_1-\rho_1u_1\right)+
\lambda_{1,2}\left(r+c_2+\sigma_2v_2+\rho_1u_1\right),\label{eq:l12dot}\\
\dot{\lambda}_{2,1} &=&\lambda_{2,2}\left(\sigma_2v_2-\rho_2u_2\right)+
\lambda_{2,1}\left(r+c_1+\sigma_1v_1+\rho_2u_2\right),
\label{eq:l21dot}\\
\dot{\lambda}_{2,2} &=&\lambda_{2,2}\left(r+c_2+\sigma_2 v_2+\rho_1u_1\right)+\lambda_{2,1}\left(\sigma_1v_1-\rho_1u_1\right)-Q_2,
\label{eq:l22dot}
\end{eqnarray}
and the transversality conditions imply that $\lambda_{k,j}(T)=0$ for all $k,j=1,2$.  Note that, as opposed to the proof of Theorem \ref{thm:nsnash}, the transversality conditions do not directly imply $\lambda_{1,2}\equiv 0$ and $\lambda_{2,1}\equiv 0$.  Make the substitution $\varphi_{k,j}(t)=\lambda_{k,j}(t)e^{-rt}/Q_k$, and the two-point boundary value problem \eqref{eq:sbvp1}--\eqref{eq:sbvp3} is then obtained.  Letting $x_1^*, x_2^*, \varphi_{1,1}, \varphi_{1,2}, \varphi_{2,1}, \varphi_{2,2}$ solve 
\eqref{eq:sbvp1}--\eqref{eq:sbvp3}, define the open-loop strategies
\begin{equation}
\uol=\frac{1}{2}Q_k\rho_kx_j^*(\varphi_{k,k}-\varphi_{k,j})e^{rt},\qquad
\vol=\frac{1}{2}Q_k\sigma_k\varphi_{k,k}(1-x_1^*-x_2^*)e^{rt}.
\label{eq:sol}
\end{equation}
and the closed-loop strategies \eqref{eq:scl}.  As in the proof of Theorem \ref{thm:nsnash}, the demonstration of the conditions \eqref{eq:gne1}--\eqref{eq:gne4} is accomplished by adding particular equations to $\Pi_k$.
First, note that
\begin{equation}
\int_0^T\frac{d}{dt}\left(Q_1(\varphi_{1,1}-\varphi_{1,2})x_1\right)\,dt+Q_1(\varphi_{1,1}(0)-\varphi_{1,2}(0))x_1^0=0.
\label{eq:add1}
\end{equation}
Then we have, using \eqref{eq:sbvp1}--\eqref{eq:sbvp3} and \eqref{eq:snash1}--\eqref{eq:snash2},
\begin{align}
\frac{d}{dt}&\left(Q_1(\varphi_{1,1}-\varphi_{1,2})x_1\right)=
Q_1(\dot{\varphi}_{1,1}-\dot{\varphi}_{1,2})x_1+Q_1(\varphi_{1,1}-\varphi_{1,2})\dot{x}_1\nonumber\\
&=Q_1 x_1\left(\varphi_{1,1}c_1-\varphi_{1,2}c_2+(\varphi_{1,1}-\varphi_{1,2})\rho_1\uo_1+  (\varphi_{1,1}-\varphi_{1,2})\rho_2\uo_2-e^{-rt} \right)     \nonumber\\
&+Q_1(\varphi_{1,1}-\varphi_{1,2})\left(\sigma_1v_1(1-x_1-x_2)+\rho_1u_1x_2-x_1(c_1+\rho_2u_2)\right)\nonumber\\
&=\rho_1Q_1(\varphi_{1,1}-\varphi_{1,2})x_2u_1+\hat{f}_{u_1}(\uo_1,v_1,\vo_1,u_2,\uo_2,\varphi_{1,1},\varphi_{1,2})-Q_1x_1e^{-rt}\nonumber\\
&=2u_1^*u_1e^{-rt}+\hat{f}_{u_1}(\uo_1,v_1,\vo_1,u_2,\uo_2,\varphi_{1,1},\varphi_{1,2})-Q_1x_1e^{-rt},\label{eq:add1a}
\end{align}
where $\hat{f}_{u_1}(\uo_1,v_1,\vo_1,u_2,\uo_2,\varphi_{1,1},\varphi_{1,2})$ is independent of $u_1$.  Then, adding the left hand side of \eqref{eq:add1} to $\Pi_1$ and using \eqref{eq:add1a}, we have for any admissible $u_1$,
\begin{eqnarray}
\Pi_1((u_1^*, v_1^*),((u_2^*, v_2^*)) &-& \Pi_1((u_1, v_1^*),((u_2^*, v_2^*))\nonumber\\&=&\int_0^T\left[Q_1x_1-{u_1^*}^2-{v_1^*}^2+2{u_1^*}^2+\hat{f}_{u_1}e^{rt}-Q_1x_1\right]e^{-rt}\,dt
\nonumber\\
&&-\int_0^T\left[Q_1x_1-{u_1}^2-{v_1^*}^2+2{u_1^*}u_1+\hat{f}_{u_1}e^{rt}-Q_1x_1\right]e^{-rt}\,dt\nonumber\\
&=&\int_{0}^T\left[{u_1^*}^2-2u_1^*u_1+u_1^2\right]e^{-rt}\,dt\nonumber\\
&=&\int_{0}^T\left(u_1^*-u_1\right)^2e^{-rt}\,dt \,\ge\, 0,\label{eq:add13}
\end{eqnarray}
which proves condition \eqref{eq:gne1}.  To prove \eqref{eq:gne2}, we use
\begin{equation}
\int_0^T\frac{d}{dt}\left(Q_1\varphi_{1,1}x_1\right)\,dt+Q_1\varphi_{1,1}(0)x_1^0=0.
\label{eq:add2}
\end{equation}
Then we have, using \eqref{eq:sbvp1}--\eqref{eq:sbvp3} and \eqref{eq:snash1}--\eqref{eq:snash2},
\begin{eqnarray}
\frac{d}{dt}\left(Q_1\varphi_{1,1}x_1\right)&=&
Q_1\dot{\varphi}_{1,1}x_1+Q_1\varphi_{1,1}\dot{x}_1\nonumber\\
&=&Q_1 x_1\left(\varphi_{1,1}\left(c_1+\sigma_1\vo_1+\rho_2\uo_2\right)
+\varphi_{1,2}\left(\sigma_2\vo_2-\rho_2\uo_2\right)-e^{-rt} \right)     \nonumber\\
&&+Q_1\varphi_{1,1}\left(\sigma_1v_1(1-x_1-x_2)+\rho_1u_1x_2-x_1(c_1+\rho_2u_2)\right)\nonumber\\
&=&\sigma_1Q_1\varphi_{1,1}(1-x_1-x_2)+\hat{f}_{v_1}(u_1,\vo_1,u_2,\uo_2,\vo_2,\varphi_{1,1},\varphi_{1,2})-Q_1x_1e^{-rt}\nonumber\\
&=&2v_1^*v_1e^{-rt}+\hat{f}_{v_1}(u_1,\vo_1,u_2,\uo_2,\vo_2,\varphi_{1,1},\varphi_{1,2})-Q_1x_1e^{-rt},\label{eq:add2a}
\end{eqnarray}
where $\hat{f}_{v_1}(u_1,\vo_1,u_2,\uo_2,\vo_2,\varphi_{1,1},\varphi_{1,2})$ is independent of $v_1$.  Proceeding as we did above, we see that \eqref{eq:add2} and \eqref{eq:add2a} imply, for any admissible $v_1$,
\begin{eqnarray}
\Pi_1((u_1^*, v_1^*),((u_2^*, v_2^*)) &-& \Pi_1((u_1^*, v_1),((u_2^*, v_2^*))\nonumber\\&=&\int_0^T\left[Q_1x_1-{u_1^*}^2-{v_1^*}^2+2{v_1^*}^2+\hat{f}_{v_1}e^{rt}-Q_1x_1\right]e^{-rt}\,dt
\nonumber\\
&&-\int_0^T\left[Q_1x_1-{u_1^*}^2-{v_1}^2+2{v_1^*}v_1+\hat{f}_{v_1}e^{rt}-Q_1x_1\right]e^{-rt}\,dt\nonumber\\
&=&\int_{0}^T\left[{v_1^*}^2-2v_1^*v_1+v_1^2\right]e^{-rt}\,dt\nonumber\\
&=&\int_{0}^T\left(v_1^*-v_1\right)^2e^{-rt}\,dt \,\ge\, 0,\label{eq:add23}
\end{eqnarray}
which proves condition \eqref{eq:gne2}.  Conditions \eqref{eq:gne3} and \eqref{eq:gne4} are shown in the same manner, completing the proof of Theorem \ref{thm:snash}.
\end{proof}

Again, computation of analytic solutions to the two-point boundary value problem in \eqref{eq:sbvp1}--\eqref{eq:sbvp3}  is generally infeasible.  In Figure \ref{fig:stp}, two different solutions are computed for the duopoly case.  As in the nontargeted simulations, the discount rate is taken to be $r=0.1$ and the profit rates $Q_1=Q_2=1$ and firm 2's effectiveness to effort ratio  towards competitors' customers is 20\% better than firm 1 ($\rho_1=1.0, \rho_2=1.2$).  Additionally, we assume that firm 1's campaign is more effective towards the market potential ($\sigma_1=1.2, \sigma_2=1.0$) and that both firms have an identical initial market share of 40\%.  The figure on the left is the solution for $c_1=c_2=0$, and the figure on the right is with $c_1=0.1, c_2=0.2$.  It is easy to see that in the absence of sales decay, the market potential tends to zero, while there is a substantial market potential in the case with nonzero sales decay.  As opposed to the nontargeted case, these two scenarios produce different firms in the market share lead.  In the absence of sales decay, initially firm 1 takes the lead as its campaign towards market potential is more effective, however as market potential shrinks, firm 2 overtakes the lead as its campaign towards customers of firm 1 is more effective, resulting in a leading long-run market share.  When sales decay is introduced however, firm 1 maintains the leading market share throughout the time horizon.  This leads to the observation that sales decay rates play an important role in the optimal strategies of a targeted advertising policy.   
\begin{figure}[ht]
\caption{Solutions to \eqref{eq:sbvp1}--\eqref{eq:sbvp3} and closed-loop strategies  \eqref{eq:scl}\label{fig:stp}}
{\subfigure[$c_1=c_2=0$]{%
\includegraphics[width=0.65\textwidth]{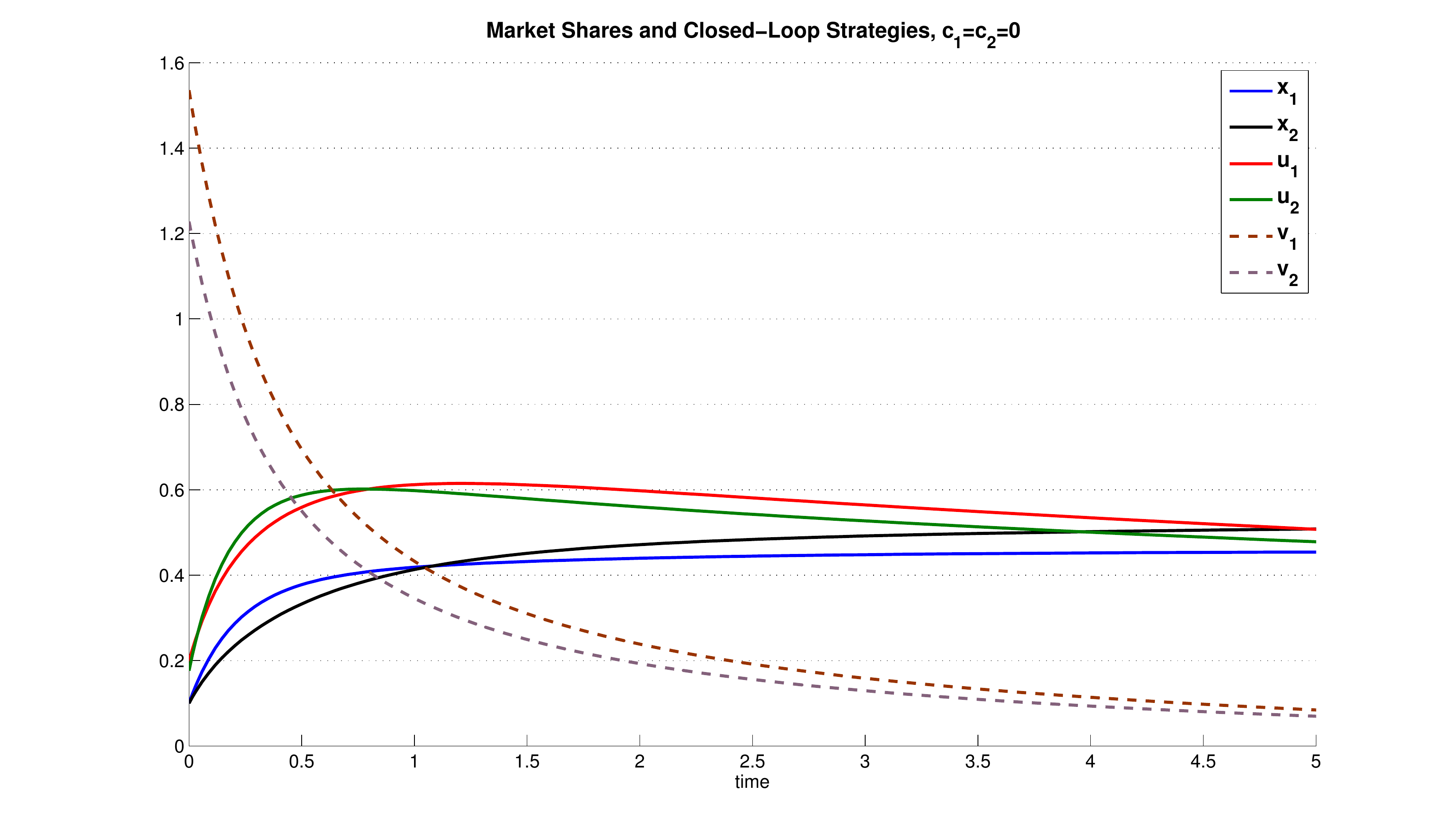}
\label{fig:nstp1}}
\quad
\subfigure[$c_1=0.1$, $c_2=0.2$]{%
\includegraphics[width=0.65\textwidth]{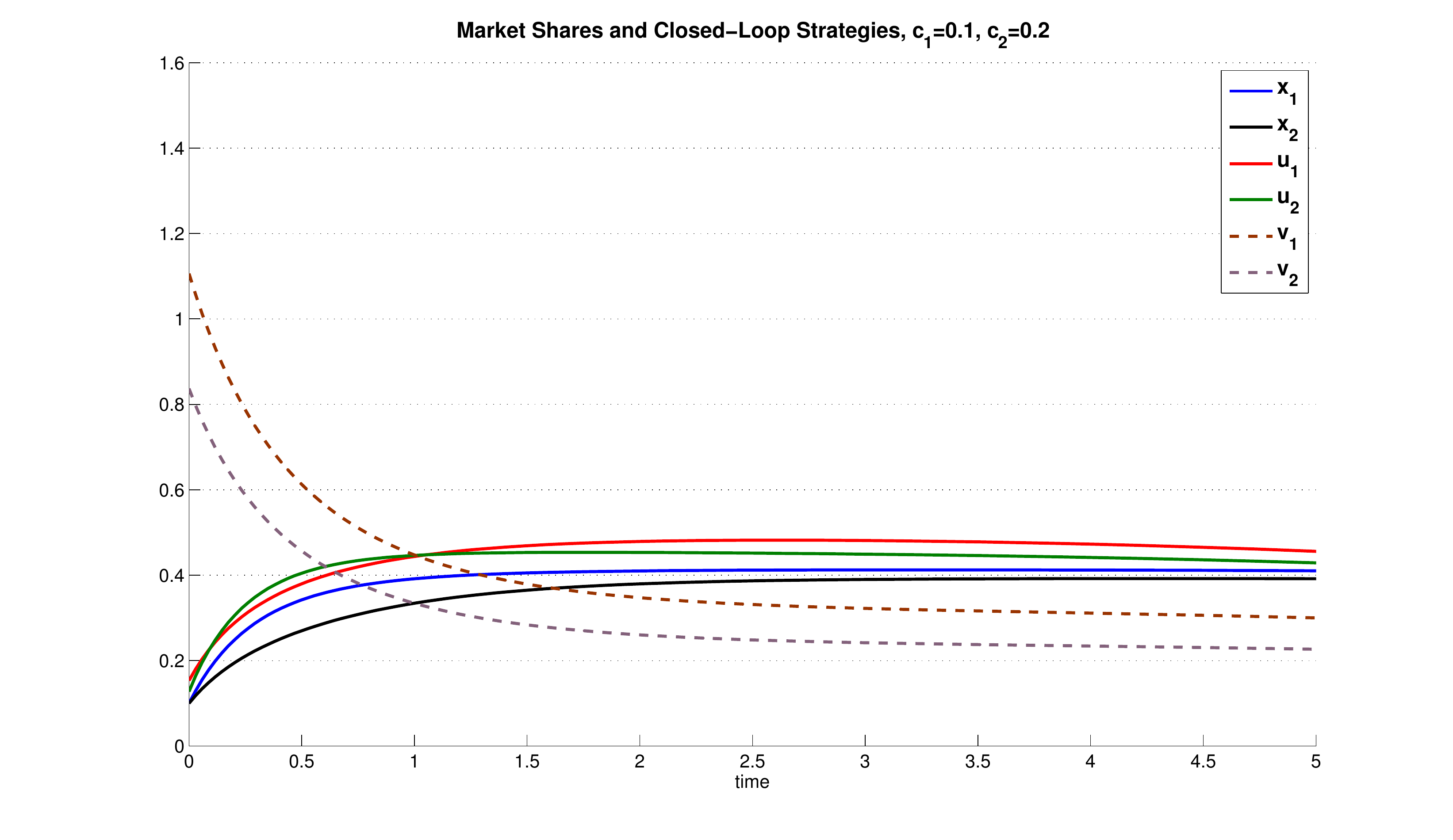}
\label{fig:nstp2}}}
\end{figure}

\subsection{Analysis of Steady States of the Targeted Model}\label{sec:ssteady}
The derivation of steady states for the targeted advertising model is significantly more complicated than the nontargeted case.  In addition to the assumptions in section \ref{sec:nssteady}, assume $\rho_k$ and $v_k$ are constant.  
The duopoly case is given by the equations
\begin{eqnarray}
\dot{s}_1&=&\sigma_1v_1m+\left(\rho_1u_1-\sigma_1v_1\right)s_2-s_1\left(c_1+\sigma_1v_1+\rho_2u_2\right),\label{eq:sduo1}\\
\dot{s}_2&=&\sigma_2v_2m+\left(\rho_2u_2-\sigma_2v_2\right)s_1-s_2\left(c_2+\sigma_2v_2+\rho_1u_1\right).\label{eq:sduo2}
\end{eqnarray}
The steady states are found by setting $\dot{s}_1=\dot{s}_2=0$ in \eqref{eq:sduo1}-\eqref{eq:sduo2}.  

\begin{lemma}\label{lem:s1}
Let $V_n=\sum_{k=1}^n \sigma_kv_k$ and 
\begin{equation}
D_2=\sigma_1v_1 \left(c_2+U_2 \right)  + \sigma_2v_2 \left(c_1+U_2 \right)  +c_1 \rho_1 u_1+ c_2  \rho_2 u_2+ c_1c_2. 
\label{eq:Ddef}
\end{equation}
Then the unique steady state of \eqref{eq:sduo1}-\eqref{eq:sduo2} with constant parameters is given by
\begin{equation}
s_1 = \frac{m\left( V_2 \rho_1 u_1+ c_2 \sigma_1 v_1\right)}{D_2},  \qquad
s_2 = \frac{m\left( V_2 \rho_2 u_2+ c_1 \sigma_2 v_2\right)}{D_2}  \label{eq:seq}.
\end{equation}
\end{lemma}

Examining the numerator of $s_1$ in \eqref{eq:seq} we see that
\begin{equation}
m\left( V_2 \rho_1 u_1+ c_2 \sigma_1 v_1\right)=m\left(\sigma_1v_1\left(\rho_1u_1+c_2\right)+\rho_1u_1\sigma_2v_2\right).
\label{eq:enum}
\end{equation}
Both sides of \eqref{eq:enum} give interpretations of the contributions firm 1's long term sales behavior.  The left side of \eqref{eq:enum} indicates that there is a contribution from firm 1's effectiveness towards customers of firm 2 times the sum of the effectiveness of both firms' campaigns towards market potential ($mV_2 \rho_1 u_1$), as well as a contribution from firm 1's effectiveness towards market potential times firm 2's cancel rate ($mc_2\sigma_1 v_1$).
The right hand side of \eqref{eq:enum} indicates that there is a contribution from the effectiveness towards market potential ($\sigma_1v_1$) times the sum of its effectiveness towards firm 2's customers ($\rho_1u_1$) and firm 2's cancellation rate ($c_2$), as well as its effectiveness towards firm 2's customers times the firm 2's effectiveness towards market potential, all multiplied by market size $m$.  


It should also be noted that for the special case $\rho_k=\sigma_k$ and $u_k=v_k$ ($k=1,2$) the equilibrium point \eqref{eq:seq} reduces to the nontargeted equilibrium \eqref{eq:nseq1}.

\begin{lemma}\label{lem:s2}
The equilibrium point \eqref{eq:seq} is stable.
\end{lemma}

\begin{proof}
The Jacobian of \eqref{eq:sduo1}-\eqref{eq:sduo2} is given by
\begin{equation}
J=\begin{bmatrix}
-\left(c_1+\sigma_1v_1+\rho_2u_2\right) & \left(\rho_1u_1-\sigma_1v_1\right)\\
\left(\rho_2u_2-\sigma_2v_2\right)&-\left(c_2(t)+\sigma_2v_2+\rho_1u_1\right)
\end{bmatrix},
\label{eq:sjac}
\end{equation}
which has eigenvalues
\begin{equation}
\lambda_{\pm} = -\frac{1}{2}\left(c_1+c_2+U_2+V_2\pm\sqrt{d_*}\right),
\label{eq:seig}
\end{equation}
where 
\begin{equation}
d_* = (c_1-c_2)^2-2(c_1-c_2)(\rho_1u_1-\rho_2v_2-\sigma_1v_1+\sigma_2v_2)+(U_2-V_2)^2.
\label{eq:dstar}
\end{equation}
It is clear that the real part of $\lambda_{+}$ is negative.  
To demonstrate that the real part of $\lambda_{-}$ is negative, note that 
\[d_* = (c_1+c_2+U_2+V_2)^2 -\left[ 4c_1c_2+c_1(\rho_1u_1+\sigma_2v_2)+c_2(\rho_2u_2+\sigma_1v_1)+U_2V_2\right].\]
Since $0\le |d_*| < (c_1+c_2+U_2+V_2)^2$, $\Re(\sqrt{d_*})<c_1+c_2+U_2+V_2$ and therefore $\Re(\lambda_{-})<0$.  Thus the equilibrium point \eqref{eq:seq} is stable.
\end{proof}

In the duopoly case, \eqref{eq:seq} implies that the equilibrium market potential is
\begin{equation}
\varepsilon = \frac{m(c_1\rho_1u_1+c_2\rho_2u_2+c_1c_2)}{D_2}.
\label{eq:seqpot}
\end{equation}
The case for $n=3$ is more complicated but the equilibrium point can be described.  
Define
\begin{equation}
D_3=\prod_{k=1}^3 c_k +\sum_{k=1}^3\left[
\sigma_kv_k\prod_{\substack{j\ne k}}^3\left(c_j+U_3\right)
+\rho_ku_kc_k\left(U_3+\DS\sum_{\substack{j\ne k}}^3 c_k\right)
\right].
\label{eq:seq3}
\end{equation}
Then the equilibrium solution is given by
\[s_1=  \left(\rho_1u_1\left(U_3V_3+\sigma_1v_1(c_2+c_3)+\sigma_2v_2c_3+\sigma_3v_3c_2\right)+\sigma_1v_1\left(\DS\sum_{\substack{j\ne k}}^3\rho_ku_kc_k+c_2c_3\right)\right)/D_3,\]
with $s_2$ and $s_3$ given similarly (for example, $s_2$ is found by exchanging all terms with subscripts of $1$ to terms with subscripts of $2$).



\section{Targeted Advertising Strategies Under a Fixed Budget}\label{sec:app}
In this section we discuss the application of the models presented in Sections \ref{sec:ns} and \ref{sec:s} to answer various strategic questions a firm may face when information identifying customers of competing firms (or the market potential) and competitors' behavior is available. 
While the  open-loop and closed-loop solutions found in Sections \ref{sec:nsnash} and \ref{sec:snash} give optimal strategies when all information about competitors' advertising policies and sales decay due to cancellation are known, 
it is often the case that firms must operate on limited information and/or in a very short time horizon.  Additionally, those open-loop and closed-loop strategies are computed in the absence of a fixed budget for advertising expenditures (as all positive $u_k$ and $v_k$ are considered to be admissible controls - a fixed advertising budget would remove these controls from the objective function).  In these cases a firm's strategy may be determined by the current best course of action based on analyzing the targeted and non-targeted models and their steady states.  

Additionally, the models here can be applied to situations in which the primary objective is to simply dominate the market share, maximize sales, or maximize rate of sales increase. For example, in the political candidate/election setting, the objective of each candidate is to have a larger market (voter) share at the time of the election.  In these cases the differential games formed by \eqref{eq:nsgame1}--\eqref{eq:nsgame2} and \eqref{eq:sgame1}--\eqref{eq:sgame2} are not as relevant and further analysis of the models are required.

For the remainder of this section, we assume that firm 1 has a constant budget amount $B$ for expenditures, which implies that effort controls $u_1$ and $v_1$ satisfy $u_1^2+v_1^2=B$, or $v_1=\sqrt{B-u_1^2}$.  When $B=1$, the control $u_1$ represents the square root of the portion of the budget that 
is allocated towards competitors' customers.

\subsection{Maximizing Rate of Increase of Sales Rate/Market Share.}
A firm or political candidate that gains the ability to discern the customers/supporters of their competitors from the market potential may wish to immediately implement an advertising strategy that will increase sales as quickly as possible in a short time window. We describe the best strategy for initial allocation towards competitors' customers to maximize the current rate of increase of sales rate or market share.
\begin{quote}
{\bf Question 1:} Given $n-1$ competitors with nontargeted or targeted advertising policies, which initial allocation of effort maximizes the instantaneous rate of sales increase for firm $k$?
\end{quote}  
The objective for firm $k$ is to maximize 
\begin{equation}
\dot{s}_k=m\sigma_k\sqrt{1-u_k^2}+\left(\rho_k u_k-\sigma_k\sqrt{1-u_k^2}\right)\sum_{\substack{j=1\\j\ne k}}^ns_j  -s_k\left(c_k+\sigma_k\sqrt{1-u_k^2}+\sum_{\substack{j=1\\j\ne k}}^n\rho_j u_j\right).
\label{eq:q51}
\end{equation}
with respect to $u_k$.
Note that firm $j$'s cancellation rates and efforts towards market potential do not affect $\dot{s}_k$ for all $j\ne k$.  Define $N$ to be the totality of the market that firm $k$ does not hold, i.e., 
\begin{equation}
N = m-s_k,
\label{eq:Ndef}
\end{equation}
and define $X$ to be the proportion of $N$ held by the competitors of firm $k$, i.e.,
\begin{equation}
X=\frac{1}{N}\left(\sum_{\substack{j\ne k}}^ns_j\right)=\frac{m-\varepsilon-s_k}{N}=\frac{N-\varepsilon}{N}.
\label{eq:Xdef}
\end{equation}
Then the allocation of effort that maximizes the rate of sales increase is given in the following theorem.

\begin{theorem}\label{thm:q5}
Let firm $k$ practice targeted advertising and assume $u_k^2+v_k^2=1$.
Then the instantaneous rate of sales increase for firm $k$ is maximized by the allocation
\begin{equation}
u_k=\frac{\rho_kX}{\sqrt{\sigma_k^2(1-X)^2+\rho_k^2X^2}},
\label{eq:q53}
\end{equation}
where $X$ is defined as in \eqref{eq:Xdef}. 
\end{theorem}

\begin{proof}  Differentiating \eqref{eq:q51} with respect to $u_k$ gives
\[\frac{\partial \dot{s}_k}{\partial u_k} = \frac{-\sigma_k u_km+\sigma_k(m-\varepsilon-s_k)\sqrt{1-u_k^2}+\rho_k(m-\varepsilon-s_k)u_k+s_k\sigma_ku_k}{\sqrt{1-u_k^2}}.\]  
We have
\[\frac{\partial \dot{s}_k}{\partial u_k} =0\qquad\Longrightarrow\qquad u_k=\frac{\rho_k(m-\varepsilon-s_k)}{\sqrt{\sigma_k^2\varepsilon^2+\rho_k^2(m-\varepsilon-s_k)^2}},\]
and since
\[\frac{\partial^2\dot{s}_k}{\partial u_k^2}=\frac{-\sigma_k\varepsilon}{(1-u_k^2)^{3/2}},\]
the second derivative test indicates that $\dot{s}_k$ is maximized there.  Letting $N=m-s_k$, we have that the optimal $u_1$ is
\[u_1=\frac{\rho_k(N-\varepsilon)}{\sqrt{\sigma_k^2\varepsilon^2+\rho_k^2(N-\varepsilon)^2}}.\]
Dividing through by $N$ and letting $X=(N-\varepsilon)/N$ so that $\varepsilon/N=1-X$ we have
\[u_1=\frac{\rho_kX}{\sqrt{\sigma_k^2(1-X)^2+\rho_k^2X^2}}.\]
\end{proof}

Theorem \ref{thm:q5} is useful in that it gives a simple formulation that allows firm $k$ to determine their best course of action at the current moment, and all it requires is knowledge of its competitors'  proportion ($X$) of the market not occupied by firm $k$ ($N$).
Dividing through by $\sigma_k$, \eqref{eq:q53} can be written as
\begin{equation}
u_k=\frac{(\rho_k/\sigma_k)X}{\sqrt{(1-X)^2+(\rho_k/\sigma_k)^2X^2}}.
\label{eq:q54}
\end{equation}
Plots of the optimal choice of $u_k$ versus competitors' non-firm $k$ market share $X$ for several choices of $\rho_k/\sigma_k$ are given in Figure \ref{fig:q5}.  Because the horizontal axis is scaled to represent the proportion of the non-firm $k$ market, for a given $\rho_k/\sigma_k$, all curves have the same profile for any competitors' share of the total market (the profile of the curves is independent of the size of the non-firm $k$ market $N$).
\begin{figure}[ht]
\caption{Values of $u_k^2$ from \eqref{eq:q54} that Maximize $\dot{s}_k$ for varying $\rho_k/\sigma_k$ \label{fig:q5}}
\includegraphics[scale=1]{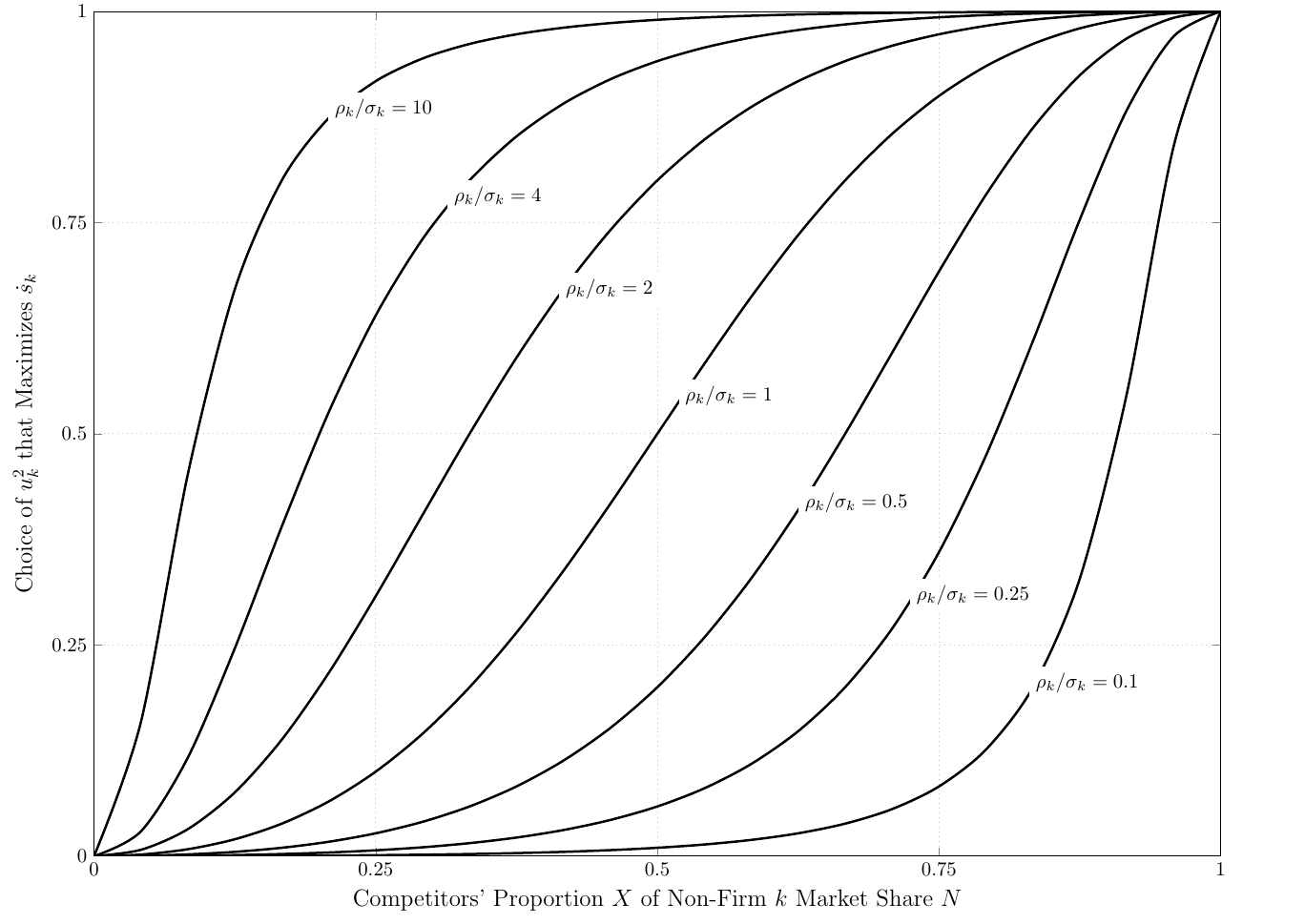}
\end{figure}
It is important to note that the horizontal axis is the proportion of the non-firm $k$ market share occupied by the competitors, and not the actual value of the competitors' market share.  Consider the context of a political candidate in an upcoming election who obtains the ability to direct advertising differently towards uncommitted voters and supporters of the opponent (candidate 2). Suppose that candidate 1 currently holds 30\% of the vote, and candidate 2 has 35\% of the vote (so that candidate 2 holds $X=50\%$ of the voters not committed to candidate 1).  If candidate 1's advertising campaign is equally effective towards uncommitted voters and supporters of candidate 2 (so $\rho_k/\sigma_k=1$), then 
\[ u_1^2=\frac{(\rho_k/\sigma_k)^2X^2}{{(1-X)^2+(\rho_k/\sigma_k)^2X^2}}=\frac{1\cdot 0.5^2}{{\left(0.5\right)^2+(1)^2\left(0.5\right)^2}}=\frac{1}{2},\]
so candidate 1 should split the campaign budget evenly across both categories.  However, if candidate 1's advertising is twice as effective at gaining uncommitted voters as it is in converting supporters of candidate 2 (so $\rho_k/\sigma_k=1/2=0.5$), then 
\[ u_1^2=\frac{(\rho_k/\sigma_k)^2X^2}{{(1-X)^2+(\rho_k/\sigma_k)^2X^2}}=\frac{0.25\cdot 0.5^2}{{\left(0.5\right)^2+0.25\left(0.5\right)^2}}=\frac{1}{5},\]
so candidate 1 should allocate only 20\% of the campaign budget towards uncommitted voters.

Alternatively, if candidates 1 and 2 are both tied with 40\% of the vote each, then candidate 2 holds $X=2/3$ of the market of voters not currently supporting candidate 1.  In this case, the above scenarios of $\rho_k/\sigma_k=1$ and $\rho_k/\sigma_k=1/2$ give optimal budget allocations of $u_1^2=80\%$ and $u_1^2=50\%$, respectively - drastically different strategies than the 30\%-35\% scenario above.   The proposition below summarizes the utility of this result.

\begin{proposition}
Under a fixed budget, the allocation of advertising effort that optimizes the rate of firm $k$'s sales increase is determined only by firm $k$'s effort to effectiveness ratios and the proportion of the available market held by firm $k$'s competitors, and is independent of sales decay rates, competitors' strategies, efforts, and effectiveness. 
\end{proposition}

\subsection{Targeted Advertising Strategies in an Oligopoly With Nontargeted Competitors}
When it is known that competing firms in an oligopoly practice nontargeted advertising, a firm that gains the ability to discern the competitors' customer base from the market potential should determine the conditions under which a targeted advertising policy will lead to an increased market share.  We assume that all firms' sales are at a steady state when the following questions are considered.   In this case we have $\sigma_2=\rho_2$ and $v_2=u_2$. With the assumption that $v_2^2+u_2^2=1$, we have that $v_2=u_2=1/\sqrt{2}$.

\begin{quote}
{\bf Question 2:} Given a single competitor with a nontargeted advertising policy, which allocation of effort will maximize steady state market share?
\end{quote}
 If expenditures, efforts, and cancellation rates are all constant, then the equilibrium sales for firms 1 and 2 are
\begin{equation}
s_1=\frac { m\left( { \sigma_1} \left( { \rho_1}\,{ u_1}+{ 
c_2} \right) \sqrt {B-{{ u_1}}^{2}}+{ \rho_1}{ \rho_2}
{ u_1}{ u_2} \right) }{ \left( { \rho_1}{ u_1}+{
 \rho_2}{ u_2}+{ c_2} \right)  \left( { \sigma_1}\sqrt {B-{{ u_1}
}^{2}}+{ \rho_2}{ u_2}+{ c_1} \right) }
\label{eq:app1b}
\end{equation}
and 
\begin{equation}
s_2=\frac {\rho_2u_2}{\rho_1u_1+\rho_2u_2+c_2}.
\label{eq:app2a}
\end{equation}
To determine the best policy for firm 1, $s_1$ is maximized with respect to $u_1$.  Assume for simplicity that $\rho_1=\rho_2=\sigma_1=1$ and $m=1$.  Then we have
\begin{equation}
s_1=\frac {\frac{1}{\sqrt{2}}u_1+ \left( { c_2}+{ u_1} \right) \sqrt {1-{{ u_1}}^{2}
}}{ \left( \frac{1}{\sqrt{2}}+{ c_2}+{ u_1}
 \right)  \left( \frac{1}{\sqrt{2}}+c_1+\sqrt {1-{{ u_1}}^{2}}
 \right) }.
\label{eq:s1vsu1}
\end{equation}
Given particular values for the cancellation rates $c_1$ and $c_2$, \eqref{eq:s1vsu1} can be maximized with respect to $u_1$ to determine the optimal allocation of effort. 
Values of $s_1$ as a function of $u_1$ for different values of $c_1$ (with $c_2=1$) and $c_2$ (with $c_1=1$) are given in Figure \ref{fig:s1c}, in which the red dot indicates the maximum of $s_1$.  In Figure \ref{fig:s1c}(a) we see that as $c_1$ increases, market share is maximized by decreasing $u_1$ (the effort towards customers of the competitor), which corresponds to an increase in effort towards the market potential.  Interestingly, Figure  \ref{fig:s1c}(b) indicates similar behavior - as firm 2's cancellation rate increases, firm 1 should decrease the effort $u_1$ towards the customers of firm 2 (thereby concentrating more effort towards the market potential) in order to maximize the market share.
\begin{figure}[ht]
\caption{Dependence of $s_1$ on $c_1$ and $c_2$ in \eqref{eq:s1vsu1}\label{fig:s1c}}
{\subfigure[Varying $c_1$ ($c_2=0.1$)]{%
\includegraphics[width=0.65\textwidth]{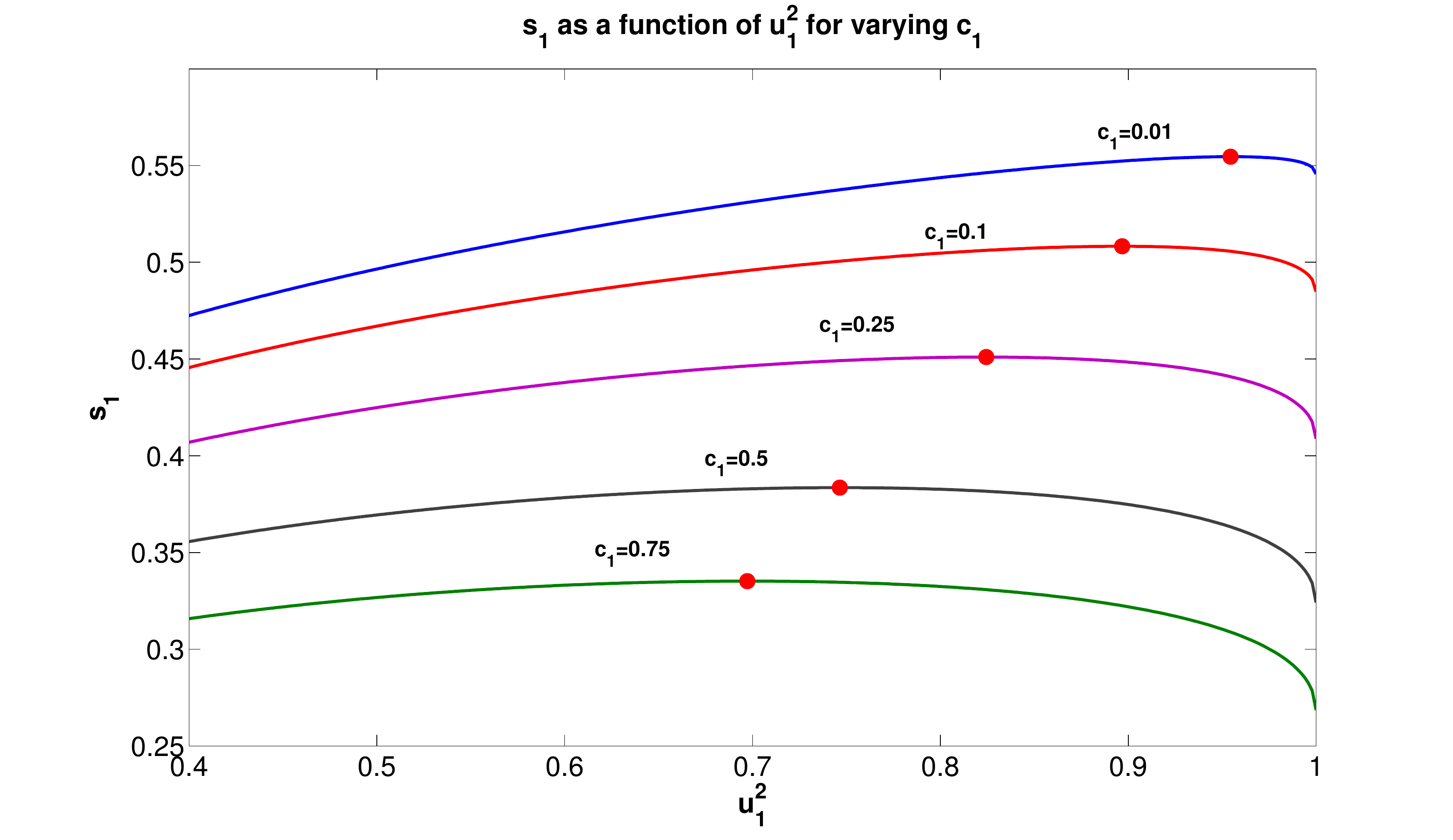}
\label{fig:s1c1}}
\quad
\subfigure[Varying $c_2$ ($c_1=0.1$)]{%
\includegraphics[width=0.65\textwidth]{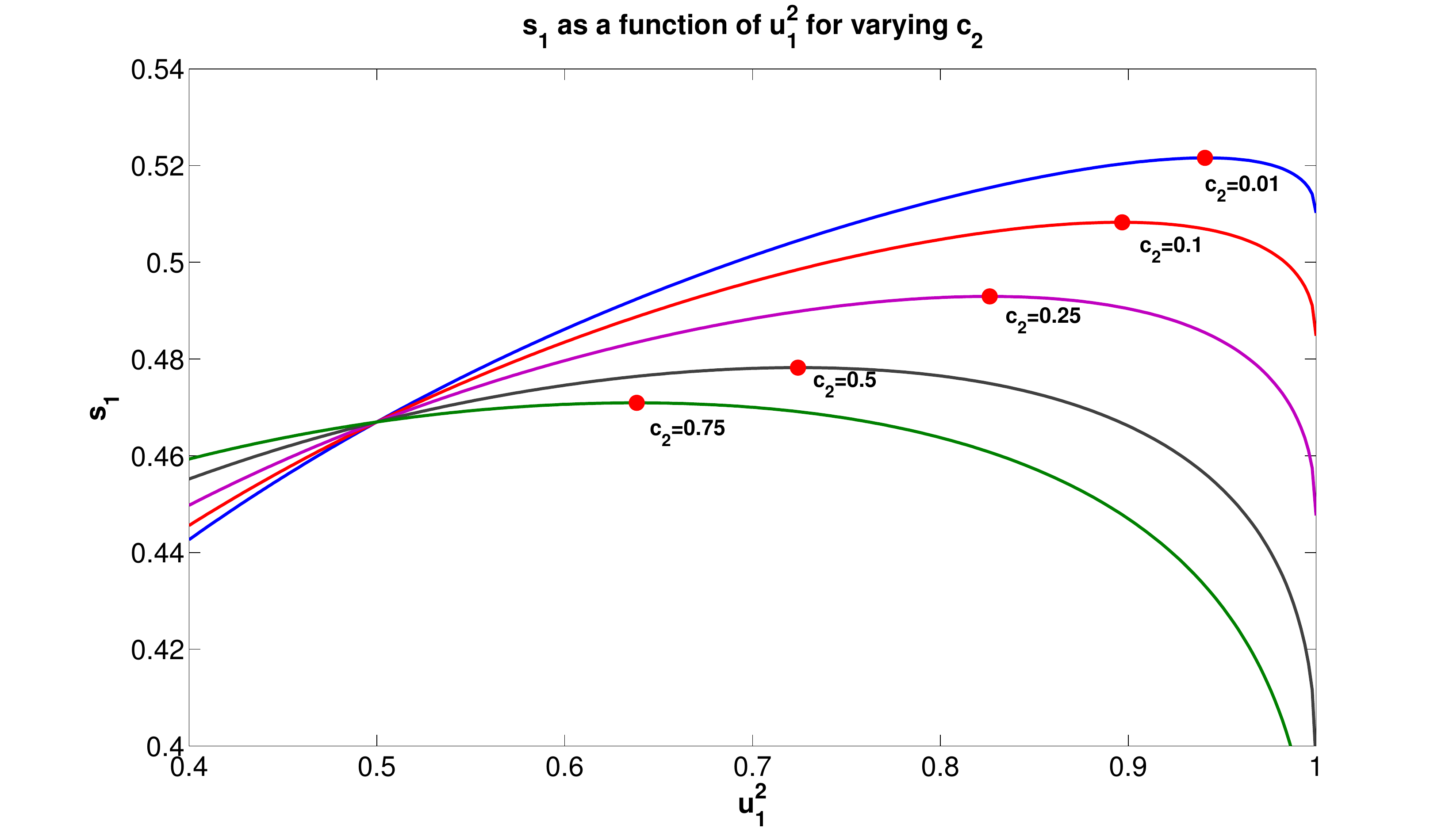}
\label{fig:s1c2}}}
\end{figure}
From Figure \ref{fig:s1c} it is evident that the cancellation rate $c_1$ has less of an impact than $c_2$ on the choice of $u_1$ that optimizes market share for firm 1.  To illustrate the dependence of the optimal $u_1$ on $c_1$ and $c_2$, Figure \ref{fig:s2c} gives a contour plot of the value $u_1$ that maximizes $s_1$ for varying $c_1$ and $c_2$.  The optimal choice of $u_1$ is determined numerically.  For example, if $c_1=0.3$ and $c_2=0.2$, then $s_1$ is maximized when $u_1^2\approx 0.76$ ($u_1\approx 0.87$), which means that to optimize steady sales, firm 1 should dedicate around 76\% of their advertising budget toward customers of firm 2.  In this case the steady state market shares for firm 1 and firm 2 are $s_1\approx0.429$ and $s_2\approx 0.397$.

\begin{figure}[ht]
\caption{Contour plot of $u_1^2$ that maximizes $s_1$ for $c_1$ and $c_2$ in \eqref{eq:s1vsu1} \label{fig:s2c}}
{
\includegraphics[width=0.75\textwidth]{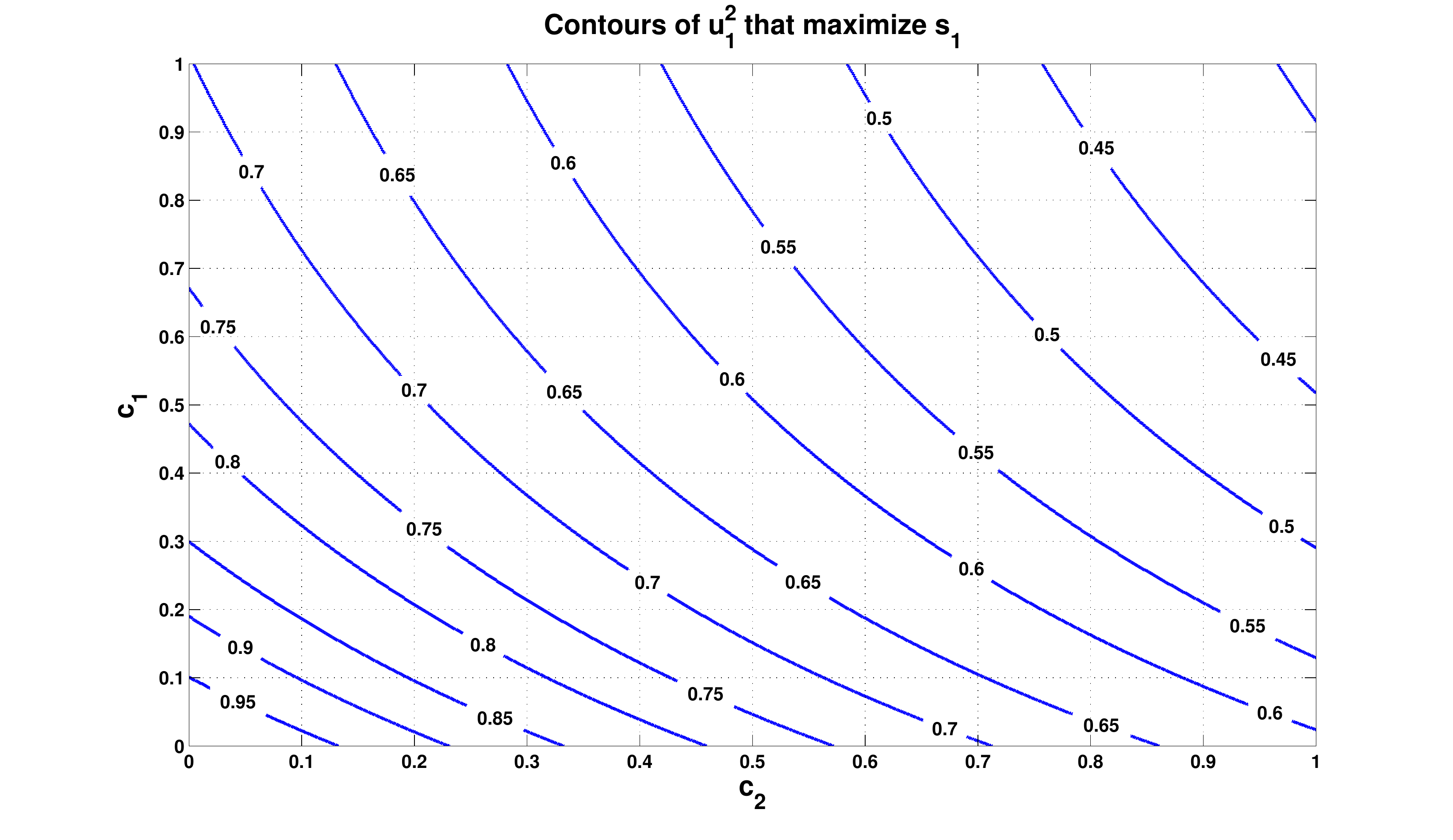}
}
\end{figure}

However, if we examine the case $c_1=c_2=0.9$, then we have that
$u_1^2\approx 0.43$ ($u_1\approx 0.65$) maximizes $s_1$.  Then the optimal (for firm 1) equilibrium market shares are
\begin{equation}
s_1\approx 0.306,\qquad s_2\approx 0.313.
\label{eq:lower}
\end{equation}
This leads to an important observation:
\begin{proposition}
When cancellation rates, advertising budgets, and effectiveness to effort ratios for each firm are all the same, the targeted strategy that maximizes sales rate/market share does not necessarily lead to higher sales than a competitor practicing nontargeted advertising.  
\end{proposition}

In other words, if a leading market share is more important than maximizing the market share, a different allocation of effort may be in order.

\begin{quote}
{\bf Question 3:} Given a single competitor with a nontargeted advertising policy, which allocation of effort will ensure a greater steady state market share than your competitor?
\end{quote}
The answer to this question (what choices of $u_1$ ensure $s_1>s_2$) varies greatly depending on the relative values of $c_1$ and $c_2$.  Subtracting \eqref{eq:app2a} from \eqref{eq:app1b} and employing
the assumptions $\rho_1=\rho_2=\sigma_1=1$, $u_2=1/\sqrt{2}$, and $v_1=\sqrt{1-u_1^2}$, then $s_1-s_2>0$ whenever
\begin{equation}
\sqrt{1-u_1^2}\left(c+u_1+\frac{1}{\sqrt{2}}\right)+\frac{1}{\sqrt{2}}
\left(u_1-c-\frac{1}{\sqrt{2}}  \right)
 > 0.
\label{eq:q4a}
\end{equation}
The region in the $u_1^2,c$-plane that satisfies \eqref{eq:q4a} (and therefore ensures $s_1>s_2$) is shaded in Figure \ref{fig:s2d}, along with  curves that identify the value of $u_1^2$ that maximizes $s_1-s_2$ and the value of $u_1^2$ that maximizes $s_1$.  Note that $u_1=1/\sqrt{2}$ will always guarantee $s_1=s_2$ as this reduces firm 1 to the same nontargeted advertising policy as firm 2.
\begin{figure}[ht]
\caption{{the shaded area indicates the choices of $u_1^2$  that ensure $s_1>s_2$ when $c_1=c_2=c$. \label{fig:s2d}}{The solid curve indicates the value of $u_1^2$ that maximizes $s_1-s_2$, while the dotted curve represents the value of $u_1^2$ that maximizes $s_1$.}}
{
\includegraphics[width=0.8\textwidth]{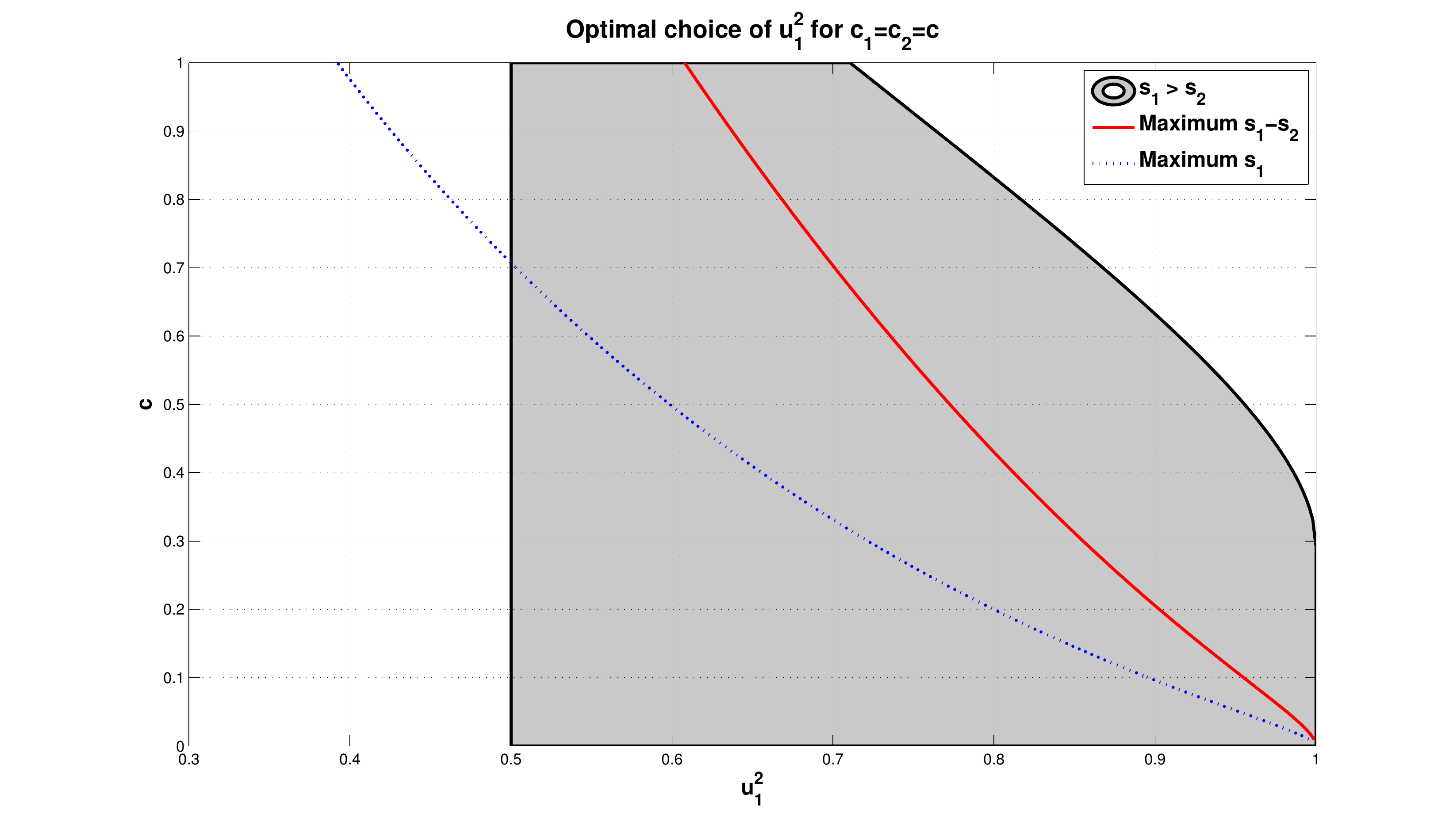}
}

\end{figure}
In relation to the example illustrated in \eqref{eq:lower}, the choice of effort $u_1^2=0.637$ ($u_1=0.798$) leads to steady market shares of
\begin{equation}
s_1\approx 0.299,\qquad s_2\approx 0.293,
\label{eq:higher}
\end{equation}
again reinforcing the conclusion that maximized market share does not imply leading market share, and vice versa.

\begin{proposition}
When employing targeted advertising with a competitor who practices nontargeted advertising with an identical budget and sales decay rate $0< c\le 1$, a strategy that ensures market share greater than your competitor is always possible.  Additionally, the strategy that maximizes market share does not coincide with the strategy that maximizes the lead in market share over your competitor.
\end{proposition}

\subsection{Extension to a Single Firm with Service/Product Tiers}\label{sec:tier}
The targeted advertising model can also be adapted to represent the sales dynamics of a single firm's internal competition among products or services. 
Consider the case that a single firm offers a product/service at several ($n$) different levels or tiers, and assume that the customers of each tier are mutually-exclusive (i.e., no single customer simultaneously contributes to the sales of more than one level/tier).  
Often the firm's objectives are twofold: to recruit new customers from the market potential, as well as to have existing customers of lower-tier products/services upgrade to a higher tier.
Assume the firm wishes to prevent downgrades in level, so higher tiers consider customers of lower tiers as potential customers, but not vice-versa. 

Let $k=1, 2, \ldots, n$ represent the $n$ tiers of product/service, with the convention that $k=1$ represents the highest tier, $k=2$ the next highest, and so on, with $k=n$ representing the entry-level tier.  The multi-tier model is based on modifying \eqref{eq:sel2} to incorporate the following assumptions.  Since all non-customers of the firm comprise the market potential, the firm (all tiers) would have a uniform advertising effort $v(t)$ towards $\varepsilon(t)$.  However, the effectiveness to effort ratios $\sigma_k$ will vary among the tiers, in particular if new customers are more likely to sign up for the entry-level tier as opposed to the highest tier.  Additionally, tier 1 would have an advertising campaign directed towards customers of tiers $2,\ldots, n$, tier 2 would advertise towards tiers $3,\ldots, n$ and so on (tier $n$ would only gain customers from the market potential).  Each tier would have its own cancellation rate.  Then the dynamics of the market share $s_k$ of tier $k$ would be modeled by
\begin{equation}
\dot{s}_k(t)=\sigma_k v(t)\underbrace{\left(m-\sum_{j=1}^ns_j(t)\right)}_{\text{market potential}} +\rho_ku_k(t)\underbrace{\sum_{j>k}^ns_j(t)}_{\text{lower tiers}}  -s_k(t)\left(c_k(t)+\underbrace{\sum_{j<k}^n\rho_ju_j(t)}_{\text{higher tiers}}\right).
\label{eq:serv1}
\end{equation}

Dropping the notational dependence upon $t$, when there are two tiers within a service/product line, we have
\begin{eqnarray}
\dot{s}_1&=& \sigma_1 v \left(m -s_1-s_2\right)+\rho_1u_1s_2-s_1c_1,\label{eq:serv2}\\
\dot{s}_2&=& \sigma_2 v \left(m -s_1-s_2\right)-s_2\left(c_2+\rho_1u_1\right).\label{eq:serv3}
\end{eqnarray}
One objective would be to minimize the equilibrium market potential, thereby ensuring the firm has as many customers as possible.  For $n=2$, the equilibrium solution of \eqref{eq:serv2}--\eqref{eq:serv3} is given by
\[s_1=\frac{mv((\sigma_1+\sigma_2)\rho_1u1+c_2\sigma_1)}{((\sigma_1+\sigma_2)v+c_1)\rho_1u_1+(c_1\sigma_2+c_2\sigma_1)v+c_1c_2},\]
and 
\[s_2=\frac{mv\sigma_2c_1}{((\sigma_1+\sigma_2)v+c_1)\rho_1u_1+(c_1\sigma_2+c_2\sigma_1)v+c_1c_2}.\]
Then the equilibrium market potential is given by
\begin{equation}
\varepsilon = \frac{mc_1(c_2+\rho_1u_1)}{((\sigma_1+\sigma_2)v+c_1)\rho_1u_1+(c_1\sigma_2+c_2\sigma_1)v+c_1c_2},
\label{eq:serv4}
\end{equation}
and can be minimized with respect to one of the parameters given particular choices for the remaining parameters.

A practical interpretation of this model would be the case where 1) the advertising budget is fixed (at $1$) so $u_1^2+v^2=1$, and 2) the effectiveness to effort ratios towards potential are the same for both tiers ($\sigma_1=\sigma_2=\sigma$).  Further, if $m=\rho_1=1$, we have that \eqref{eq:serv4} gives an equilibrium market potential of
\begin{equation}
\varepsilon = \frac{c_1(c_2+u_1)}{(\sigma(c_1+c_2+2u_1)\sqrt{1-u_1^2}+c_1(c_2+u_1))}.
\label{eq:serv5}
\end{equation}
In this case, an advertising strategy that minimizes market potential (maximizing market share) can be described.  

\begin{theorem}\label{thm:tier1}
Let $\sigma_1=\sigma_2=\sigma=1$, $\rho_1=1$, $m=1$, and assume $u_1^2+v^2=1$.  If $c_1\ge c_2$,  then $u_1=0$ minimizes the equilibrium market potential \eqref{eq:serv5}, and if $c_1<c_2$, then the equilibrium market potential is minimized for 
\begin{equation}
u_1^*=\frac{10c_2^2-6c_1c_2+\Delta_1^{2/3}-4c_2\Delta_1^{1/3}}{6\Delta_1^{1/3}},
\label{eq:serv6}
\end{equation}
where 
\[\Delta_1 = 36{ c_1}c_2^{2}-28c_2^{3}-
54{ c_1}+54{ c_2}+ 6\sqrt{\Delta_0},\]
and \[\Delta_0=6\left( { c_1}-{ c_2} \right)  \left( c_2^{5}+2{ c_1}
 c_2^{4}+ \left( c_1^{2}+14 \right)  c_2^{3}-18{
 c_1} c_2^{2}-{\frac {27{ c_2}}{2}}+{\frac {27{ c_1}
}{2}} \right).\]
\end{theorem}

\begin{proof}  Differentiating \eqref{eq:serv5} with respect to $u_1$ gives
\[\frac{\partial \varepsilon}{\partial u_1} = {\frac {{ c_1} \left( 2{{ u_1}}^{3}+4{ c_2}{{ u_1}}^{2
}+{ c_2} \left( { c_1}+{ c_2} \right) { u_1}+{ c_1}-{ 
c_2} \right) }{\sqrt {1-{{ u_1}}^{2}} \left(  \left( { c_1}+{ 
c_2}+2{ u_1} \right) \sqrt {1-{{ u_1}}^{2}}+{ c_1} \left( {
 c_2}+{ u_1} \right)  \right) ^{2}}}.\]  This is defined for all $c_1>0, c_2>0$, and $0\le u_1\le 1$.
Let 
\[p(u_1)=2{{ u_1}}^{3}+4{ c_2}{{ u_1}}^{2
}+{ c_2} \left( { c_1}+{ c_2} \right) { u_1}+{ c_1}-{ 
c_2}.\]
Then the zeros of $\partial \varepsilon/\partial u_1$ are the zeros of $p$.  We consider two cases.\\
{\bf Case 1:} ($c_1\ge c_2$)  Then $p>0$ for all $0\le u_1 \le 1$ and thus $\varepsilon$ is at a minimum when $u_1=0$.\\
{\bf Case 2:} ($c_1<c_2$)  Since $p(0)<0$ and $p(1)>0$, the Intermediate Value Theorem implies the existence of a $0<u_1^*<1$ such that $p(u_1^*)=0$, and an application of Rolle's theorem shows $u_1^*$ is unique as $p'(u_1)>0$ everywhere in $0<u_1<1$.  In this case $u_1^*$ is given by
\eqref{eq:serv6} and is guaranteed to minimize $\varepsilon$.
\end{proof}

Figure \ref{fig:serv1} gives a contour map of the values of $u_1$ that minimize $\varepsilon$ for different choices of $c_1$ and $c_2$.
\begin{figure}[ht]
\caption{Value of $u_1$ that minimizes $\varepsilon$ in \eqref{eq:serv5} \label{fig:serv1}}
{
\includegraphics[width=0.75\textwidth]{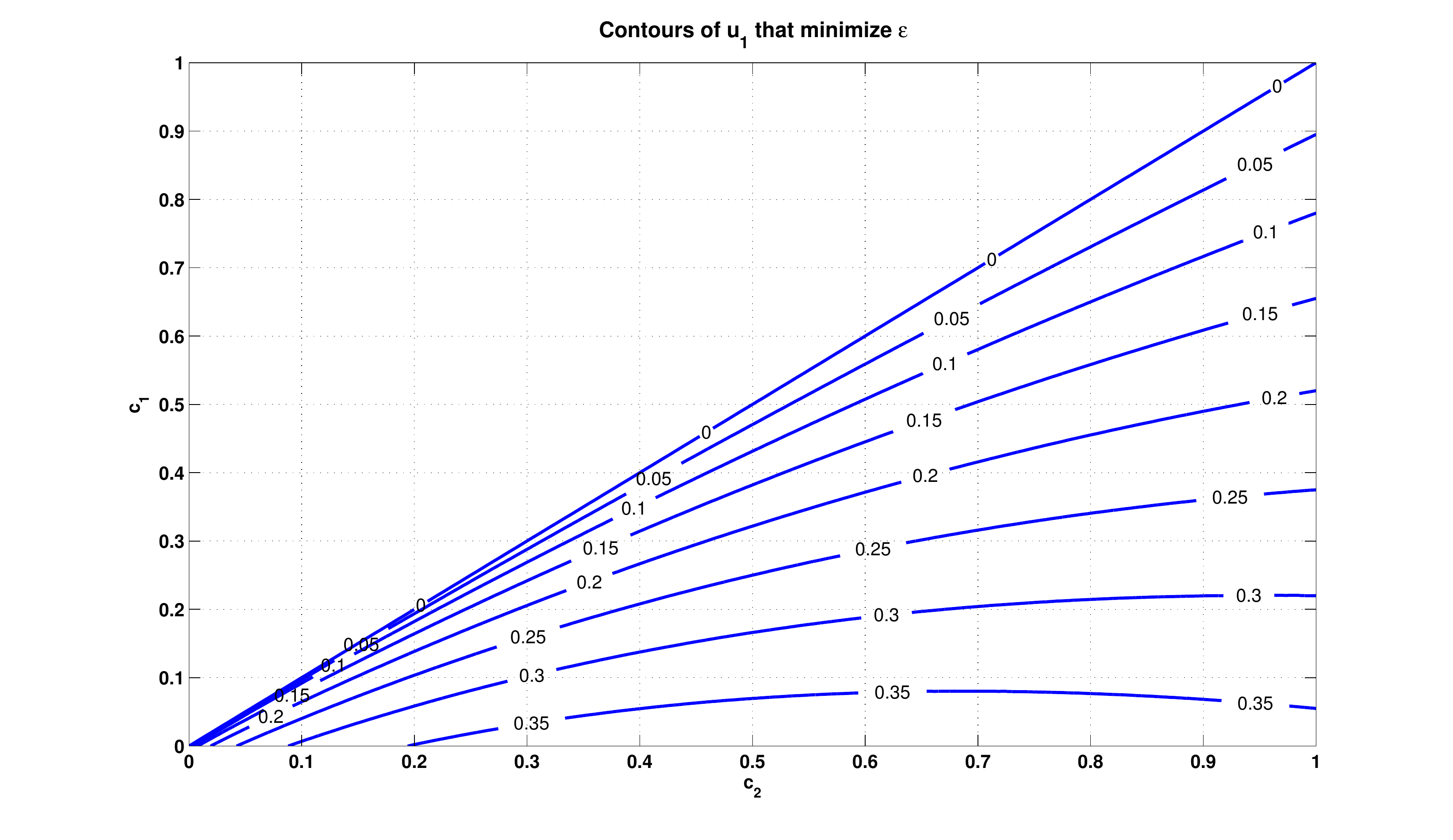}
}
\end{figure}
For example, if the cancellation rate for tier 1 is $c_1=0.1$ and for tier 2 is $c_2=0.2$, then an allocation of around 6.25\% of the advertising budget (as $u_1\approx 0.25$) targeting upgrades for customers of the second tier will minimize the market potential (maximize the total market share of the firm).

One of the main implications of Theorem \ref{thm:tier1} is that if customers of the higher tier are more likely to cancel service than those of the lower tier ($c_1\ge c_2$), the market potential is minimized when all advertising effort is dedicated toward the market potential, implying that any overall gain in market share from encouraging upgrades is lost through the higher cancellation rate for those customers who do upgrade.

\begin{proposition}
The advertising strategy that minimizes total market potential for a firm that offers two tiers of product/service depends on the relative effectiveness of campaigns directed towards existing customers and market potential and the sales decay rates for each tier.  In particular, when the effectiveness of both campaigns are identical and the  sales decay rate is larger for the higher tier, all advertising effort should be directed towards the market potential and no effort should be directed towards enticing existing customers to upgrade. 
\end{proposition}

\section{Summary}\label{sec:sum}
Existing oligopoly advertising models were modified to incorporate market share decrease by cancellation and subsequently extended to allow for differing advertising policies for the market potential and competitor customer base.  Closed-loop Nash equilibrium strategies were found for both the nontargeted and targeted models.  Steady state market shares were identified in the case of constant parameters, and several applications of the models were employed to address strategic questions when the ability to discern competitors' customers from market potential is  available.
The analysis points to several interesting conclusions:
\begin{itemize}
\item When operating with a fixed budget, the allocation of effort towards competitors' customers and market potential that optimizes the current rate of sales increase is independent of whether competitors employ targeted or nontargeted advertising as well as sales decay rates. 
\item When employing a targeted advertising approach, the distribution of effort that maximizes steady state market share does not always lead to a market share greater than a  competitor practicing nontargeted advertising.
\item For a single firm with multiple tiers of product/service, the allocation of effort that minimizes market potential is heavily dependent on the relative sales decay rates of the tiers.
\end{itemize}
It is anticipated that the targeted model and the observations made here may be of use to firms as they gain the ability to group advertising targets.
Further work in this direction could include:
\begin{itemize}
\item Extension to the case of differing efforts for different competitors.
\item Extension to competing firms with multiple product lines (in the spirit of \cite{Fruchter2001}).
\item Incorporating more sophisticated models of effort and effectiveness of customer retention activities.
\item Extension of the single-firm multi-tier analysis to multiple firms each offering several tiers of product/service.
\end{itemize}


%
%
%






\bibliographystyle{spmpsci}      
\bibliography{refs}   

%
%

\end{document}